\documentclass[11pt,a4paper]{amsart}
\usepackage{geometry}           
\usepackage{graphicx} 
\usepackage{mathrsfs} 
\usepackage{amsfonts} 
\usepackage{enumitem}
\usepackage{amsthm}
\usepackage{amssymb}
\usepackage{amsmath}
\usepackage{hyperref}
\usepackage{mathtools}
\usepackage{tikz}
\usepackage[
  backend=biber,
  style=alphabetic,
  doi=false
]{biblatex}

\usepackage{cleveref}

\usepackage{marginnote}
\usepackage{subcaption}

\usetikzlibrary{matrix}
\usetikzlibrary{positioning}
\usepackage{tikz-cd}
\usetikzlibrary{decorations.pathreplacing}
\usepackage{mathabx}

\numberwithin{equation}{section}

\newtheorem{theorem}[equation]{Theorem}

\newtheorem{corollary}[equation]{Corollary}
\newtheorem{lemma}[equation]{Lemma}
\newtheorem{proposition}[equation]{Proposition}

\newtheorem{claim}[equation]{Claim}

\theoremstyle{definition}
\newtheorem{definition}[equation]{Definition}

\newtheorem*{convention}{Convention}

\newtheorem{example}[equation]{Example}

\theoremstyle{remark}
\newtheorem{remark}[equation]{Remark}
\newtheorem{question}{Question}

\AddToHook{env/theorem/begin}{\crefalias{equation}{theorem}}
\AddToHook{env/lemma/begin}{\crefalias{equation}{lemma}}
\AddToHook{env/corollary/begin}{\crefalias{equation}{corollary}}
\AddToHook{env/proposition/begin}{\crefalias{equation}{proposition}}
\AddToHook{env/claim/begin}{\crefalias{equation}{claim}}
\AddToHook{env/definition/begin}{\crefalias{equation}{definition}}
\AddToHook{env/assumption/begin}{\crefalias{equation}{assumption}}
\AddToHook{env/convention/begin}{\crefalias{equation}{convention}}
\AddToHook{env/observation/begin}{\crefalias{equation}{observation}}
\AddToHook{env/example/begin}{\crefalias{equation}{example}}
\AddToHook{env/remark/begin}{\crefalias{equation}{remark}}
\AddToHook{env/question/begin}{\crefalias{equation}{question}}
\AddToHook{env/conjecture/begin}{\crefalias{equation}{conjecture}}

\crefname{theorem}{Theorem}{Theorems}
\crefname{lemma}{Lemma}{Lemmas}
\crefname{proposition}{Proposition}{Propositions}
\crefname{definition}{Definition}{Definitions}
\crefname{corollary}{Corollary}{Corollaries}
\crefname{example}{Example}{Examples}
\crefname{assumption}{Assumption}{Assumptions}
\crefname{remark}{Remark}{Remarks}
\crefname{conjecture}{Conjecture}{Conjectures}
\newcommand{\nc}{\newcommand}
\nc{\dmo}{\DeclareMathOperator}
\dmo{\ra}{\rightarrow}
\dmo{\N}{\mathbb{N}}
\dmo{\F}{\mathbb{F}}
\dmo{\Z}{\mathbb{Z}}
\dmo{\R}{\mathbb{R}}
\dmo{\UD}{\mathsf{UD}}
\dmo{\D}{\mathsf{D}}

\DeclareMathOperator{\id}{\mathrm{Id}}

\usepackage{stackengine}
\stackMath

\newcommand{\EO}{%
  \mathord{\stackon[-9.4pt]{\mathcal{G}{}}{\widehat{\phantom{\mathcal{G}{}}}}}%
}

\newcommand{\EOW}{\EO^{\raisebox{-0.6ex}{$\scriptstyle\mathfrak S$}}}

\dmo{\Stab}{Stab}
\dmo{\B}{Big}
\dmo{\GBG}{\mathbb{B}}
\dmo{\Sat}{Sat}
\dmo{\diam}{diam}
\dmo{\MCG}{MCG}
\dmo{\Min}{Min}
\dmo{\Isom}{Isom}
\dmo{\syl}{syl}
\dmo{\dist}{\mathsf{d}}

\DeclareMathOperator{\CAT}{CAT}

\newcommand{\calC}{\mathcal{C}}
\newcommand{\calV}{\mathcal{V}}
\newcommand{\calE}{\mathcal{E}}

\newcommand{\bfF}{\mathbf{F}}

\newcommand{\bbA}{\mathbb{A}}

\newcommand{\frG}{\mathfrak{G}}

\newcommand{\frS}{\mathfrak{S}}

\newcommand\loops\ell
\renewcommand\bar\overline
\renewcommand\tilde\widetilde

\makeatletter
\newcommand{\introsubsection}{%
  \@startsection{subsection}{\@M}%
  \z@{.5\linespacing\@plus.7\linespacing}{-.5em}%
  {\normalfont\bfseries}%
}
\makeatother

\begin{document}
\title[From algebraic orthogonality to RAAG embedding obstructions in HHGs]{From algebraic orthogonality to RAAG embedding obstructions in hierarchically hyperbolic groups}
\author{Sangrok Oh}
\address{Innovation Center for MathScience Research \& Education, Pusan National University, Korea}
\email{SangrokOh.math@gmail.com}
\author{Jihoon Park}
\address{Department of Mathematics Education, Kyungpook National University, Korea}
\email{ttsiug@knu.ac.kr}

\begin{abstract}
We extend the Kim--Koberda obstruction arguments for embeddings of right-angled Artin groups (RAAGs) to a class of hierarchically hyperbolic groups (HHGs). To this end, we isolate algebraic conditions on HHGs relating domain orthogonality to the decomposition and commutation of group elements. Under these conditions, we obtain a clique-graph obstruction to RAAG embeddings into such HHGs using the \emph{expanded core graph}, which we introduce as a counterpart of both the extension graphs of RAAGs and the disjointness graphs of curves on surfaces. When the HHG structure has rank at most two, the expanded core graph itself gives a complete embedding criterion. 
\end{abstract}

\maketitle
\enlargethispage{3pt}
\setcounter{tocdepth}{1}
\tableofcontents

\section{Introduction}
A \emph{right-angled Artin group (RAAG)} is the group associated to a finite simplicial graph \(\Gamma\), defined by
\[\bbA(\Gamma)=\langle v\in \calV(\Gamma)\,\mid\, [v,w]=1 \text{ whenever } \{v,w\}\in \calE(\Gamma)\rangle,\]
where \(\calV(\Gamma)\) and \(\calE(\Gamma)\) are the vertex and edge sets, respectively; the graph \(\Gamma\) is called the \emph{defining graph} of \(\bbA(\Gamma)\).
A recurring problem in geometric group theory is to determine which RAAGs embed into a given ambient group and whether the resulting subgroups are undistorted. The principal motivating examples are RAAGs themselves and mapping class groups.
For a RAAG \(\bbA(\Gamma)\), the extension graph \(\Gamma^e\), introduced by Kim--Koberda, records commutation among conjugates of standard generators: an induced embedding \(\Lambda\leq\Gamma^e\) gives an embedding \(\bbA(\Lambda)\leq\bbA(\Gamma)\), while an arbitrary RAAG embedding yields the clique-graph obstruction \(\Lambda\leq(\Gamma^e)_k\), where \((\Gamma^e)_k\) is the clique graph of \(\Gamma^e\) \cite{KK13}; see \Cref{def:extension graph,Thm:KK RAAG emb Thm}. In mapping class groups, collections of mapping classes supported on suitably positioned subsurfaces generate RAAGs after passing to sufficiently large powers \cite{Kob12,CLM12}, while an embedding \(\bbA(\Lambda)\leq\MCG(S)\) yields a clique-graph obstruction in the disjointness graph of essential curves \cite{KKObstruction}.

RAAGs and mapping class groups are both fundamental examples of hierarchically hyperbolic groups (HHGs), introduced in \cite{BHS17I,BHS19II}. An HHG \((G,\frS)\) comes with hyperbolic spaces indexed by domains, together with nesting and orthogonality relations among those domains; see \Cref{Subsection:HHG} for background. 
In this paper, we show that, under suitable structural assumptions on \((G,\frS)\), every RAAG embedding into \(G\) factors through an intermediate RAAG after raising the images of the standard generators to a common positive power.
We define the \emph{expanded core graph} of \((G,\frS)\) and show that it contains the extension graph of this intermediate RAAG as an induced subgraph. Applying the Kim--Koberda clique-graph obstruction to the resulting embedding between RAAGs gives a common generalization of the two obstructions above.

\introsubsection{From arbitrary embeddings to intermediate RAAGs}

In our previous work \cite{OP25}, we investigated the constructive direction of the RAAG embedding problem for HHGs. We considered a geometrically irredundant collection of axial elements, each fully supported on a single unbounded domain. Under a virtual commutation condition, we showed that there is a positive integer \(k\) such that, for every sufficiently large multiple \(N\) of \(k\), their \(N\)-th powers generate the RAAG prescribed by orthogonality of their supporting domains. Under an additional bounded-orbit condition, the resulting RAAG subgroup is quasi-isometrically embedded. 
For rigidly fully supported elements, the bounded-orbit condition is automatic, while the virtual commutation condition follows after passing to a common positive power. 
Here, an element fully (rigidly fully, resp.) supported on an unbounded domain not only fixes the unbounded domains orthogonal to its supporting domain but also induces an action with uniformly bounded orbits (a trivial action, resp.) on each associated hyperbolic space.
See \Cref{Def:FSE,Prop:SupportedAxialRAAGs} for the relevant terminology and the general statement.

Here we study the converse problem: starting with an arbitrary injective homomorphism from a RAAG into an HHG, we ask whether the embedding, after replacing the images of the standard generators by suitable powers, factors through a RAAG subgroup generated by geometrically controlled axial elements.
A natural approach is suggested by the reductions to intermediate RAAGs used in the Kim--Koberda obstruction arguments for RAAGs and mapping class groups \cite{KK13,KKObstruction}.
In the RAAG setting, the images of standard generators are decomposed into pure factors and their commutation data are organized into an intermediate RAAG.
A closely related reduction appears in the mapping class group setting, where powers of the images are decomposed into pure components that generate an intermediate RAAG whose defining graph is realized in the disjointness graph of curves.
However, the HHG axioms do not directly provide the decomposition and commutation properties needed for such a reduction.

To address this problem, we introduce a class \(\Xi\) of HHG structures inspired by the framework of Abbott--Behrstock \cite{AB23}.
Roughly speaking, the assumptions defining \(\Xi\) ensure that suitable powers of infinite-order elements decompose into products of pairwise commuting \emph{rigidly} fully supported axial elements, and that rigidly fully supported elements on orthogonal domains have commuting positive powers.
This class contains the rich-family HHG structures on compact special groups, including RAAGs, but standard HHG structures on mapping class groups do not in general belong to it; see \Cref{Prop:Xiclass,Rem:WhyOmega}.
To include mapping class groups in our framework, we therefore introduce a broader class \(\Omega\) whose axioms capture the relevant features of \(\Xi\) directly in terms of fully supported axial elements; we describe these classes more precisely after stating the obstruction results.

Combined with the RAAG embedding theorem from \cite{OP25}, the assumptions defining \(\Omega\) suffice to extend the Kim--Koberda intermediate-RAAG reduction to the HHG setting, yielding the following factorization.

\begin{theorem}[Intermediate RAAG, \Cref{Prop:IntermediateRAAG}]\label{Thm:IntroIntermediateRAAG}
Let \((G,\frS)\) be an HHG belonging to \(\Omega\), and let \(\bbA(\Lambda)\) be a RAAG. Suppose that \(\phi\colon\bbA(\Lambda)\to G\) is an injective homomorphism. Then there exist a positive integer \(N\) and a subgroup \(M\le G\), isomorphic to a RAAG and generated by fully supported axial elements, such that the assignment \(v\mapsto\phi(v)^N\) for each vertex \(v\in \calV(\Lambda)\) extends to an injective homomorphism \(\bbA(\Lambda)\rightarrow M\).
If \((G,\frS)\in\Xi\), then \(M\) may moreover be chosen to be generated by rigidly fully supported axial elements and to be undistorted in \(G\).
\end{theorem}

The construction also applies to any finite collection of elements of \(G\): for some \(N>0\), the \(N\)-th power of each element is a product of powers of pairwise commuting standard generators of a single intermediate RAAG; see \Cref{Cor:FiniteSetRAAGs}.
This is an HHG analogue of Koberda's envelopment result for mapping class groups \cite[Corollary~1.2]{Kob12}.
As a consequence, groups virtually in \(\Omega\) satisfy Wise's power alternative: suitable powers of any two elements either commute or freely generate a group of rank two.
For structures in \(\Xi\), the resulting two-generator subgroup is undistorted in either case; see \Cref{Cor:PowerAlternative}.
Wise asked whether the power alternative holds for CAT(0) groups and automatic groups \cite[Question~2.7]{Bestvina04PA}.

The class \(\Xi\) also contains the one-domain HHG structures on hyperbolic groups. Moreover, both \(\Xi\) and \(\Omega\) satisfy the following permanence properties.

\begin{proposition}[Permanence, \Cref{Lem:FiniteIndexXiOmega,Prop:Xiclass,Prop:OmegaPermanence}]
The classes \(\Xi\) and \(\Omega\) are preserved under restriction to finite-index subgroups. Both classes are also closed under finite direct products equipped with the rectangular product structures of \Cref{Lem:RectangularProductStructure} and under the standard relatively hyperbolic construction.
\end{proposition}

To obtain combinatorial obstructions from the algebraic factorization in \Cref{Thm:IntroIntermediateRAAG}, we associate an analogue of the extension graph to the chosen HHG structure.

\begin{convention}
For readability, we state the main results in the introduction for groups equipped directly with an HHG structure in \(\Omega\) or \(\Xi\). In the body, we prove the corresponding statements more generally for groups having a finite-index subgroup equipped with such a structure; see \Cref{Def:VirtualHHGStructure}.
\end{convention}

\introsubsection{The expanded core graph and Kim--Koberda-type obstructions}

For an HHG \((G,\frS)\), the \emph{expanded core graph} \(\EOW\) is constructed from the minimal elements of the subposet of unbounded domains.
Each such domain contributes one vertex when its associated hyperbolic space is a quasi-line, and countably infinitely many vertices otherwise; for structures in \(\Omega\), these vertices correspond to the distinct axial directions supported on the domain. Adjacency records orthogonality; see \Cref{Def:ExpandedCoreGraph,Lem:AxialDirections}.
For a RAAG with a rich family containing every singleton subgraph of its defining graph, the minimal unbounded domains have quasi-line coordinate spaces, so no expansion is needed; see \Cref{Prop:RAAGExpandedCore}. In a general HHG, however, a single minimal unbounded domain may support several inequivalent axial directions. The orthogonality graph of minimal unbounded domains, which we call the \emph{core graph}, therefore need not contain enough information to encode the extension graphs of the intermediate RAAGs.

If \(M\cong\bbA(\Gamma)\) is an intermediate right-angled Artin subgroup of \(G\) obtained in \Cref{Thm:IntroIntermediateRAAG}, then \Cref{lem:intermediate RAAG} gives \(\Gamma^e\le\EOW\). Combining this embedding with \Cref{Thm:IntroIntermediateRAAG} and the Kim--Koberda obstruction gives our main combinatorial result. 

\begin{theorem}[RAAG obstruction, \Cref{Thm:RAAGObstruction}]
\label{Thm:IntroObstruction}
Let \((G,\frS)\in\Omega\), and let \(\Lambda\) be a finite graph. If \(\bbA(\Lambda)\) embeds into \(G\), then \(\Lambda\) embeds into the clique graph \((\EOW)_k\) of the expanded core graph \(\EOW\) as an induced subgraph.
\end{theorem}

The clique graph appears because the resulting embedding \(\bbA(\Lambda)\rightarrow\bbA(\Gamma)\) in \Cref{Thm:IntroIntermediateRAAG} can be a general RAAG embedding, rather than one known to arise from an induced embedding of \(\Lambda\) into \(\Gamma^e\). 
The \emph{rank} of an HHG structure is the maximal cardinality of a collection of pairwise orthogonal unbounded domains. If \((G,\frS)\) has rank at most \(2\), then the defining graph of every intermediate RAAG is triangle-free, so the stronger Kim--Koberda extension-graph criterion applies. This yields a complete criterion in terms of the expanded core graph.

\begin{theorem}[Rank-two criterion, \Cref{Thm:RankTwoCriterionHHG}]
\label{Thm:IntroRankTwo}
Let \((G,\frS)\in\Omega\) have rank at most \(2\). Then, for every finite graph \(\Lambda\), the following are equivalent:
\begin{enumerate}
\item \(\bbA(\Lambda)\) embeds into \(G\);
\item \(\Lambda\) embeds into \(\EOW\) as an induced subgraph.
\end{enumerate}
\end{theorem}

The implication from an induced graph embedding to a group embedding holds in every structure belonging to \(\Omega\), without a rank restriction: axial elements representing the vertices of the induced subgraph generate the prescribed RAAG after passing to suitable common positive powers. 
The converse also holds without a rank restriction for certain classes of graphs \(\Lambda\), including forests and complements of linear forests; see \Cref{Cor:PositiveCases}.

The obstruction theorem also has a chromatic consequence. Rich-family HHG structures on compact special groups become colourable after restriction to a suitable finite-index subgroup, so their expanded core graphs have finite chromatic number. Combining this with Erd\H{o}s's construction of finite graphs with arbitrarily large girth and chromatic number gives the following.

\begin{corollary}[Chromatic obstruction for compact special groups, \Cref{Cor:ColourableExamples}]\label{Cor:IntroChromaticObstruction}
Let \(G\) be a compact special group. Then, for every positive integer \(M\), there exists a finite graph \(\Lambda_M\) of girth at least \(M\) such that \(\bbA(\Lambda_M)\) cannot embed into \(G\).
\end{corollary}

After the first draft of this paper appeared, we applied our framework to graph braid groups in joint work with Byung Hee An \cite{AOP26}, providing a concrete application of our obstruction to compact special groups. In that work, we constructed explicit HHG structures using factor systems encoded by subgraphs and particle distributions. For graph \(2\)-braid groups, we obtained an induced-subgraph criterion for RAAG embeddings in terms of the expanded core graph and showed that every RAAG subgroup has a bipartite defining graph.

\introsubsection{The classes \texorpdfstring{\(\Xi\)}{Xi} and \texorpdfstring{\(\Omega\)}{Omega}, and the main examples}

We now describe our class \(\Xi\) more precisely, noting that it differs from the class denoted by the same symbol in \cite{AB23}.
For each domain \(U\), let \(\mathbf F_U\) denote its standard nesting factor, with induced HHS structure \((\mathbf F_U,\frS_U)\). We define \(G_U\leq\Stab_G(U)\) by requiring its elements to fix each unbounded domain \(V\) orthogonal to \(U\) and to act trivially on \(\calC V\). 
Our \emph{\(\mathbf F_U\) stabilizers property} requires the natural induced quasi-action of \(G_U\) on \((\mathbf F_U,\frS_U)\) to define an HHG structure on \(G_U\) for every domain \(U\).
Together with the \emph{orthogonal decomposition property}, which requires suitable positive powers of infinite-order elements to decompose into products of pairwise commuting rigidly fully supported axial elements, this stabilizers property defines \(\Xi\); see \Cref{Def:OurGUandFU,Def:OrthogonalDecomposition}.
The stabilizers property also ensures that, whenever \(U\) and \(V\) are orthogonal unbounded domains, finite-index subgroups of \(G_U\) and \(G_V\) commute elementwise; see \Cref{Lem:VirtualCommutation}.

Abbott--Behrstock define their stabilizers using geometric realizations of the standard factors, whereas our subgroups are defined directly from the prescribed coordinate action and require no additional choice of geometric fibers. Under our stabilizers property, suitable positive powers of their commuting decomposition factors belong to our subgroups. Consequently, a structure in their class which also satisfies our stabilizers property belongs to \(\Xi\). We explain the precise finite-index comparison in \Cref{Rem:ComparisonWithAB} and \Cref{Appendix:ComparisonWithAB}.

The geometric stabilizers property alone does not imply that the action on the induced index set has finitely many orbits. In \Cref{Appendix:ComparisonWithAB}, we give a counterexample to this implication in \cite[Lemma~3.4]{AB23} and a separate proof of the finite-order conclusion for the complementary factor in \cite[Lemma~3.9]{AB23}.

The axioms of \(\Omega\) concern the existence of fully supported axial elements, decompositions into such elements, and their commutation properties, and are chosen so that \(\Xi\subseteq\Omega\); see \Cref{Def:Omega}. A key point is that every fully supported axial element in a structure belonging to \(\Xi\) has a positive power that is rigidly fully supported; see \Cref{Prop:XiContainedInOmega}.
The relation of these classes to those considered in \cite{OP25} is explained in \Cref{Rem:XiandP_2}.

The following result relates expanded core graphs to disjointness graphs for mapping class groups and extension graphs for RAAGs.

\begin{proposition}[Expanded core graphs in the motivating examples]\label{Prop:IntroExamples}
The following hold.
\begin{enumerate}
\item {\rm(\Cref{Prop:MCGinOmega,Prop:MCGExpandedCore})}
Let \(S\) be a connected finite-type orientable surface, possibly with boundary, with \(\chi(S)<0\). Then the standard HHG structure on \(\MCG(S)\) belongs to \(\Omega\). Its expanded core graph is naturally identified with the disjointness graph of simple closed curves on \(S\) that do not bound a disk or a once-punctured disk.

\item {\rm(\Cref{Prop:RAAGExpandedCore})}
Let \(G=\bbA(\Gamma)\) be equipped with a rich-family HHG structure. Then there is an induced embedding \(\EOW\le\Gamma^e\). If the rich family contains every singleton subgraph of \(\Gamma\), then \(\EOW\cong\Gamma^e\).
\end{enumerate}
\end{proposition}

Thus our obstruction takes the form of the Kim--Koberda clique-graph obstructions for RAAGs and mapping class groups \cite{KK13,KKObstruction}; see \Cref{Rem:CurveGraphConventions} for the boundary and low-complexity conventions in the latter case. For mapping class groups, \Cref{Cor:ColourableExamples} also recovers the chromatic obstruction of \cite[Theorem~1.2]{KKObstruction}.

More broadly, these results raise questions about the scope of the classes \(\Xi\) and \(\Omega\), how much information about RAAG subgroups is encoded by the expanded core graph, when the clique graph in the general obstruction can be avoided, and how the expanded core graph depends on the chosen HHG structure. We collect these questions in \Cref{Section:FurtherQuestions}.

\introsubsection{Outline of the paper}

In \Cref{Section:Preliminaries}, we recall the necessary background on HHGs, axial elements, RAAGs, extension graphs, and rich-family HHG structures on compact special groups. In \Cref{Section:Orthogonality}, we introduce the orthogonal stabilizers and relate them to rigid full support, define the classes \(\Xi\) and \(\Omega\), prove that \(\Xi\subsetneq\Omega\), and discuss the motivating examples. In \Cref{Section:IntermediateRAAG}, we prove the intermediate RAAG factorization, extend the construction to finite collections of elements, and deduce the power alternative. In \Cref{Section:Criterion}, we define the expanded core graph, describe it for mapping class groups and RAAGs, and embed the extension graphs of intermediate RAAGs into it. In \Cref{Section:Obstructions}, we derive the general obstruction, the rank-two criterion, and the chromatic obstructions. We conclude the main text with further questions in \Cref{Section:FurtherQuestions}.
The proofs of the permanence properties are collected in \Cref{Appendix:OmegaPermanence}, while \Cref{Appendix:ComparisonWithAB} compares our stabilizers and structural assumptions with those of Abbott--Behrstock and discusses cofiniteness of the induced actions.

\subsection*{Acknowledgements}
The second author was supported by the Samsung Science and Technology Foundation under Project Number SSTF-BA2202-03.

\section{Preliminaries}\label{Section:Preliminaries}
Throughout this paper, all graphs are simplicial. For two graphs \(\Lambda\) and \(\Gamma\), we write \(\Lambda\le \Gamma\) if \(\Lambda\) is isomorphic to an induced subgraph of \(\Gamma\). For a finite graph \(\Lambda\) and a group \(G\), we write \(\bbA(\Lambda)\le G\) if \(G\) contains a subgroup isomorphic to the right-angled Artin group \(\bbA(\Lambda)\) with defining graph \(\Lambda\).
Whenever we refer to a quasi-isometric embedding between finitely generated groups, the groups are equipped with word metrics with respect to finite generating sets.
For a metric space \(Y\), we denote its metric by \(\dist_Y\). For nonempty subsets \(A,B\subseteq Y\), we put
\[\dist_Y(A,B)=\inf\{\dist_Y(a,b)\mid a\in A,\ b\in B\}.\]
We identify points with singleton sets and write \(\operatorname{Haus}_Y\) for Hausdorff distance. For nonempty bounded subsets \(A,B,C\subseteq Y\), we will use
\[\dist_Y(A,C)\leq\dist_Y(A,B)+\diam_Y(B)+\dist_Y(B,C).\]

\subsection{Hierarchical hyperbolicity}\label{Subsection:HHG}
We recall the HHS terminology and standard product regions needed below before turning to HHG structures. For the full definition of an HHS, see \cite[\S1]{BHS19II}; for a largely self-contained exposition, see \cite{CRHK}.

\begin{definition}[HHS terminology and basic properties]\label{Def:HHS}
A \emph{hierarchically hyperbolic space} \((\mathcal X,\frS)\) is a quasi-geodesic metric space \(\mathcal X\) equipped with a \emph{hierarchically hyperbolic structure} indexed by \(\frS\). The set \(\frS\) is called the \emph{index set}, and its elements are called \emph{domains}.
We fix a sufficiently large constant \(E\geq 1\), independent of the domains, bounding the HHS constants and the constants in the consequences listed below.
\begin{enumerate}
\item\label{Item:CoordinateandDF} For each \(U\in\frS\), there is an \(E\)-hyperbolic space \(\calC U\), together with an \((E,E)\)-coarsely Lipschitz projection map $\pi_U:\mathcal X\to 2^{\calC U}$ with nonempty values of diameter at most \(E\); furthermore, we assume that $\pi_U$ is $E$-coarsely surjective, as permitted by \cite[Proposition~1.16]{DHS17}. Following \cite[Notation~1.2]{BHS19II}, the metric on \(\calC U\) is usually denoted by \(\dist_U=\dist_{\calC U}\). 
The domain \(U\) is called \emph{unbounded} if \(\calC U\) is unbounded.

We often regard \(\pi_U(x)\) as the \emph{\(U\)-coordinate} of \(x\). The distance formula \cite[Theorem~4.5]{BHS19II} expresses distances in \(\mathcal X\), up to uniform multiplicative and additive error, in terms of the large projection distances in the coordinate spaces \(\calC U\).

\item\label{Item:Relations}
The index set \(\frS\) is equipped with nesting, orthogonality, and transversality. Any two distinct domains are either \(\sqsubseteq\)-comparable, orthogonal, or transverse, and these possibilities are mutually exclusive. For \(U,V,W,T\in\frS\), the relations have the following properties.
\begin{itemize}
\item[(Nesting):]
a partial order \(\sqsubseteq\). Every collection of pairwise \(\sqsubseteq\)-comparable domains has cardinality at most \(E\); this property is called \emph{finite complexity}. Moreover, \(\frS\) has a unique \(\sqsubseteq\)-maximal domain, which we usually denote by \(S\).

\item[(Orthogonality):] a symmetric, irreflexive relation \(\bot\). If \(V\sqsubseteq U\) and \(U\bot W\), then \(V\bot W\). If \(U\sqsubsetneq T\), and there exists a domain \(V\sqsubsetneq T\) with \(V\bot U\), then there exists a domain \(W\sqsubsetneq T\), called a \emph{container for the orthogonal complement of \(U\) in \(T\)}, such that every domain \(V\sqsubsetneq T\) satisfying \(V\bot U\) is nested into \(W\). 

\item[(Transversality):] a symmetric relation \(\pitchfork\), where \(U\pitchfork V\) precisely when \(U\) and \(V\) are neither \(\sqsubseteq\)-comparable nor orthogonal.
\end{itemize}

\item\label{Item:RelativeProjection} (Relative projections)
For distinct non-orthogonal domains \(U,V\in\frS\), there is a \emph{relative projection} \(\rho_V^U\colon\calC U\rightarrow 2^{\calC V}\).
If \(U\pitchfork V\) or \(U\sqsubsetneq V\), then the relative projection \(\rho_V^U\) is a coarsely constant map whose image has diameter at most \(E\). In these cases, we also write \(\rho_V^U\) for its bounded image. If \(U\sqsubsetneq V\), then \(\rho_U^V\) has values of uniformly bounded diameter. 

\item\label{Item:Consistency} (Consistency) 
For all \(x\in\mathcal X\), the projections satisfy the consistency inequalities: 
\begin{itemize}
\item If \(U\pitchfork V\), then $\min\{\dist_{U}(\pi_U(x),\rho^V_U),\,\dist_{V}(\pi_V(x),\rho^U_V)\}\leq E$.
\item If \(U\sqsubsetneq V\), then $\min\{\dist_{V}(\pi_V(x),\rho^U_V),\, \diam_{\calC U}(\pi_U(x)\cup \rho^V_U(\pi_V(x))) \}\leq E$.
\end{itemize}
If \(U\sqsubsetneq V\), then $\dist_W(\rho^U_W,\rho^V_W)\le E$ whenever $W\in\frS$ satisfies either $V\sqsubsetneq W$, or $V\pitchfork W$ and $W\not\bot U$. 
Similarly, if \(U\bot V\), then by \cite[Lemma~1.5]{DHS17}, after increasing \(E\) if necessary, $\dist_W(\rho^U_W,\rho^V_W)\le 2E$ whenever both \(\rho^U_W\) and \(\rho^V_W\) are bounded relative-projection sets. 

\item\label{Item:Realization} (Realization and partial realization)
Let \(\mathfrak T\subseteq\frS\). A \emph{partial tuple indexed by \(\mathfrak T\)} is a collection \(\mathbf b=(b_U)_{U\in\mathfrak T}\), where each \(b_U\subseteq\calC U\) is nonempty and the coordinate diameters are uniformly bounded. When \(\mathfrak T=\frS\), we call \(\mathbf b\) a tuple. For \(\kappa\geq0\), such a partial tuple is \emph{\(\kappa\)-consistent} if each coordinate has diameter at most \(\kappa\) and the nested and transverse consistency inequalities in Item~\eqref{Item:Consistency} hold with constant \(\kappa\) whenever both indices belong to \(\mathfrak T\).

The realization theorem \cite[Theorem~3.1]{BHS19II} states that, for every \(\kappa\), there exists \(\theta=\theta(\kappa)\) such that every \(\kappa\)-consistent tuple \(\mathbf b\) indexed by \(\frS\) is realized by some \(x\in\mathcal X\), in the sense that \(\dist_{\calC U}(\pi_U(x),b_U)\leq\theta\) for every \(U\in\frS\). Moreover, such a realization point is coarsely unique. Thus, up to uniformly bounded ambiguity, points of \(\mathcal X\) may be viewed as uniformly consistent tuples of coordinates.

We will also use the following consequence of the partial realization axiom \cite[Definition~1.1(8)]{BHS19II}: if \(U_1,\dots,U_k\in\frS\) are pairwise orthogonal and \(p_i\in\pi_{U_i}(\mathcal X)\), then there exists \(x\in\mathcal X\) such that \(\pi_{U_i}(x)\) is uniformly close to \(p_i\) for every \(i\).
\end{enumerate}
\end{definition}

\subsubsection{Standard product regions}
Each domain of an HHS \((\mathcal X,\frS)\) determines two standard factors and an associated product region; we briefly recall the construction, referring to \cite[\S5B]{BHS19II} and \cite[\S1.3]{DHS17} for details.
For \(U\in\frS\), write
\[\frS_U=\{V\in\frS\mid V\sqsubseteq U\}\qquad\text{and}\qquad\frS_U^\bot=\{V\in\frS\mid V\bot U\}.\]
Fix a sufficiently large consistency constant \(\kappa\), which we suppress from the notation.
We denote by \(\mathbf F_U\) and \(\mathbf E_U\) the sets of \(\kappa\)-consistent partial tuples indexed by \(\frS_U\) and \(\frS_U^\bot\), respectively, and call them the \emph{standard nesting factor} and the \emph{standard orthogonality factor}.
Their elements are called \emph{nested partial tuples} and \emph{orthogonal partial tuples}, following \cite[Definitions~5.8--5.9]{BHS19II}.

The standard product-region construction gives a coarse realization map
\[\phi_U\colon\mathbf F_U\times\mathbf E_U\longrightarrow\mathcal X,\]
whose image is the \emph{standard product region} \(\mathbf P_U\).
The coordinates of \(\phi_U(\mathbf b,\mathbf e)\) indexed by \(\frS_U\) and \(\frS_U^\bot\) are prescribed, up to a uniform error, by \(\mathbf b\) and \(\mathbf e\), respectively; the remaining coordinates are uniformly close to the bounded relative projections associated to \(U\).
Fixing \(\mathbf e\in\mathbf E_U\) gives a realization map from \(\mathbf F_U\) to \(\mathcal X\), and fixing \(\mathbf b\in\mathbf F_U\) similarly gives a realization map from \(\mathbf E_U\) to \(\mathcal X\); see \cite[Construction~5.10]{BHS19II}.

We equip \(\mathbf F_U\) and \(\mathbf E_U\) with metrics for which these realization maps are uniform quasi-isometries onto their images with the subspace metric from \(\mathcal X\).
With the \(\ell_1\) product metric associated to these factor metrics, the map \(\phi_U\) is a uniform quasi-isometry onto \(\mathbf P_U\).
The nesting factor carries the standard induced HHS structure \((\mathbf F_U,\frS_U)\); see \cite[Proposition~5.11]{BHS19II}.
We use \(\mathbf E_U\) as the orthogonality factor without choosing a separate HHS index set for it.

We distinguish these metric spaces of partial tuples from their \emph{geometric models}: concrete spaces coarsely identified with the standard factors in a way compatible with their coordinates.
Relative to the ambient HHS \((\mathcal X,\frS)\) and the map \(\phi_U\), a chosen \emph{geometric \(\mathbf F_U\)-fiber} is an actual subset \(F\subseteq\mathcal X\) at uniformly bounded Hausdorff distance from \(\phi_U(\mathbf F_U\times\{\mathbf e\})\) for some \(\mathbf e\in\mathbf E_U\).
Geometric fibers are used in the comparison with Abbott--Behrstock in \Cref{Appendix:ComparisonWithAB}; our stabilizers are defined directly from the coordinate action.

\subsubsection{Hierarchically hyperbolic groups}
We use the formulation of HHG structures with equivariant projections described in \cite[\S2]{PS23}.

\begin{definition}[HHG structure]\label{Def:HHG}
Let \(G\) be a finitely generated group equipped with a word metric.
A \emph{hierarchically hyperbolic group structure}, or an \emph{HHG structure}, on \(G\) consists of an HHS structure \((G,\frS)\), together with domain and coordinate actions of \(G\) satisfying the following conditions:
\begin{enumerate}
\item
The group \(G\) acts on \(\frS\), preserving nesting, orthogonality, and transversality, with finitely many orbits.
\item
For each \(g\in G\) and \(U\in\frS\), there is an isometry \(g^\diamond(U)\colon\calC U\rightarrow\calC(gU)\), and these isometries satisfy \((gh)^\diamond(U)=g^\diamond(hU)\circ h^\diamond(U)\).
\item
The projections and bounded relative projection sets are equivariant:
\[g^\diamond(U)\pi_U(x)=\pi_{gU}(gx)\qquad\text{and}\qquad g^\diamond(V)\rho_V^U=\rho_{gV}^{gU},\]
where the first equality holds for all \(g,x\in G\) and \(U\in\frS\), and the second holds whenever \(U\sqsubsetneq V\) or \(U\pitchfork V\).
\end{enumerate}
For simplicity, we denote this HHG structure by \((G,\frS)\).
When the meaning is clear, we suppress the superscript and write \(g\) for the coordinate isometries \(g^\diamond(U)\).
In particular, saying that \(g\) acts on \(\calC U\) includes the requirement that \(gU=U\).
\end{definition}

\begin{remark}[on \Cref{Def:HHG}]
\begin{enumerate}[wide]
\item One may equivalently start with an HHS \((\mathcal X,\frS)\) equipped with a proper and cobounded uniform quasi-action of \(G\), together with domain and coordinate actions as above, allowing a uniform error in the equivariance conditions.
Pulling back along an orbit map and choosing equivariant projections and bounded relative projection sets as in \cite[\S2.1]{DHS20} gives an HHG structure on \(G\), with the same domain and coordinate actions. Thus \Cref{Def:HHG} gives the usual class of HHGs.
We use such an ambient HHS as a \emph{geometric model} of the HHG structure when specified; otherwise, we work with \(G\) equipped with a word metric.
\item Since \(G\) has finitely many orbits in \(\frS\), the bounded coordinate spaces have a uniform diameter bound \(B\). In the HHG setting, we choose the consistency constant \(\kappa\) used for the standard factors so that \(\kappa\geq B\).
\end{enumerate}
\end{remark}

A finitely generated group admitting an HHG structure is called a hierarchically hyperbolic group.
Since the choice of structure plays an essential role in this paper, whenever we refer to an \emph{HHG}, we understand a specific HHG structure to be fixed, and when emphasizing the group itself, we use the term \emph{underlying group}.

An HHG structure on a finite-index subgroup \(H\le G\) need not extend to \(G\); indeed, \(G\) need not admit any HHG structure \cite[Corollary~4.5]{PS23}. We therefore use the following terminology.

\begin{definition}\label{Def:VirtualHHGStructure}
A \emph{virtual HHG structure} on a group \(G\) is an HHG structure \((H,\frS)\) on a finite-index subgroup \(H\le G\). If \(\mathscr C\) is a class of HHG structures, we say that \(G\) is \emph{virtually in \(\mathscr C\)} if \(G\) admits a virtual HHG structure \((H,\frS)\) with \((H,\frS)\in\mathscr C\).
\end{definition}

\subsection{Axial elements}

We recall the dynamics of individual elements in an HHG \((G,\frS)\). For \(g\in G\), fix \(x\in G\) and define the \emph{bigset} of \(g\) by
\[\B(g)=\left\{U\in\frS \mid\diam_{\calC U}\bigl(\pi_U(\langle g\rangle x)\bigr)=\infty\right\}.\]
This definition is independent of the choice of \(x\). Occasionally, we use a subscript, as in \(\B_G(g)\), to emphasize the underlying group \(G\).
The elements of \(\B(g)\) are called the \emph{active domains} of \(g\).
We say that \(g\) is \emph{elliptic} if its orbits in \(G\) are bounded, and \emph{axial} if the orbit map $n\mapsto g^n x$ is a quasi-isometric embedding for some \(x\in G\). 
The results of \cite[Proposition~6.4 and Lemma~6.7]{DHS17}, \cite[Theorem~3.1]{DHS20}, and \cite[Lemma~1.21]{AB23} give the axial--elliptic dichotomy for elements of an HHG and describe their behavior on their active domains. We use the following precise formulation, recorded in \cite[Proposition~2.19]{OP25}.

\begin{proposition}[Axial--elliptic dichotomy]\label{Prop:AxialEllipticDichotomy}
For \(g\in G\), the following are equivalent: \(g\) is elliptic, \(\B(g)=\varnothing\), and \(g\) has finite order.

If \(g\) has infinite order, then \(g\) is axial. Moreover,
\begin{itemize}
\item \(g\) permutes \(\B(g)\), which is a nonempty finite pairwise orthogonal subset of \(\frS\);
\item there exists \(M=M(\frS)>0\) such that, for every \(U\in\B(g)\), the element \(g^M\) fixes \(U\) and acts loxodromically on \(\calC U\).
\end{itemize}
\end{proposition}

We will also use the following existence result, which produces an axial element whose unique active domain is the maximal one.

\begin{proposition}[{\cite[Theorem~9.14]{DHS17}}]\label{Lem:ExistenceOfAxialElement}
If $S$ is the $\sqsubseteq$-maximal domain and \(\diam(\calC S)=\infty\), then \(G\) contains an axial element $g$ with $\B(g)=\{S\}$.
\end{proposition}

We now recall the terminology and the main RAAG embedding theorem from \cite{OP25} that will be used throughout the paper.

\begin{definition}[Supported axial elements and geometric irredundancy]\label{Def:FSE}
An axial element \(g\in G\) is \emph{fully supported on} an unbounded domain \(U\in\frS\) if \(\B(g)=\{U\}\) and, for every unbounded domain \(V\bot U\), the element \(g\) fixes \(V\) and the induced action of \(\langle g\rangle\) on \(\calC V\) has orbits of uniformly bounded diameter, with a bound independent of the initial point but allowed to depend on \(g\) and \(V\).
If, in addition, \(g\) acts trivially on \(\calC V\) for every such \(V\), then \(g\) is said to be \emph{rigidly fully supported on \(U\)}.

A finite collection of axial elements \(g_1,\dots,g_m\) is \emph{geometrically irredundant} if, for all \(i\neq j\), either \(\B(g_i)\neq\B(g_j)\), or there exists \(U\in\B(g_i)\cap\B(g_j)\) such that positive powers of \(g_i\) and \(g_j\) fixing \(U\) have disjoint limit sets in \(\partial\calC U\).
\end{definition}

\begin{remark}[Support and the ambient HHS]\label{Rem:SupportAndAmbientHHS}
Full support and rigid full support depend only on the hierarchical relations, the coordinate spaces, and the prescribed domain and coordinate actions.
Indeed, by equivariance and \Cref{Prop:AxialEllipticDichotomy}, a domain belongs to \(\B(g)\) if and only if some positive power of \(g\) fixes it and acts loxodromically on its coordinate space.
Thus these notions are unchanged when the ambient HHS is replaced while retaining these data, in particular when the HHG structure is pulled back along an orbit map to \(G\).
\end{remark}

\begin{proposition}[{\cite[Theorem~3.9]{OP25}}]\label{Prop:SupportedAxialRAAGs}
Let \(g_1,\dots,g_m\in G\) be a geometrically irredundant collection of axial elements. Suppose that \(g_i\) is fully supported on an unbounded domain \(U_i\in\frS\) for each \(i\).
Let \(\Gamma\) be the graph with vertex set \(\{v_1,\dots,v_m\}\), where \(v_i\) and \(v_j\) are adjacent if and only if \(U_i\bot U_j\).

Suppose that there exists a positive integer \(k_1\) such that \([g_i^{k_1},g_j^{k_1}]=1\) whenever \(U_i\bot U_j\). Then there exists \(D>0\) such that, for every integer \(d\geq D\), the assignment \(v_i\mapsto g_i^{dk_1}\) extends to an injective homomorphism \(\bbA(\Gamma)\rightarrow G\). Equivalently, \(\langle g_1^{dk_1},\dots,g_m^{dk_1}\rangle\cong\bbA(\Gamma)\).

Suppose, in addition, that there exist a positive integer \(k_2\) and a constant \(Q\geq 0\) such that, for every \(i\), the action of the subgroup \(\langle g_j^{k_2} \mid U_j\bot U_i\rangle\) on \(\calC U_i\) has all orbits of diameter at most \(Q\).
Then, setting \(k=\operatorname{lcm}(k_1,k_2)\) and increasing \(D\) if necessary, the homomorphism defined by \(v_i\mapsto g_i^{dk}\) is a quasi-isometric embedding for every \(d\geq D\).
\end{proposition}

\subsection{RAAGs and extension graphs}\label{Subsection:RAAG}

The main objects in the study of RAAG embedding problems are \emph{extension graphs} and their clique graphs; the latter are particularly useful for embedding obstructions.

\begin{definition}\label{def:extension graph}
\begin{enumerate}
\item For a finite (simplicial) graph \(\Gamma\), its \emph{extension graph}, denoted \(\Gamma^e\), is the graph whose vertex set is \(\calV(\Gamma^e)=\{v^g=gvg^{-1}\mid v\in \calV(\Gamma),\ g\in \bbA(\Gamma)\}\), where two vertices are adjacent if and only if they commute in \(\bbA(\Gamma)\).

\item For a (possibly infinite) graph \(\Lambda\), its \emph{clique graph}, denoted \(\Lambda_k\), is the graph whose vertices are the nonempty finite cliques of \(\Lambda\), with two vertices adjacent if and only if the corresponding cliques are contained in a common clique of \(\Lambda\).
\end{enumerate}
\end{definition}

We first note the following elementary rigidity property of vertices of the extension graph, which will be used repeatedly below.

\begin{lemma}[Powers of conjugates of standard generators]\label{Lem:ExtensionGraphPowerRigidity}
Let \(\bbA(\Gamma)\) be a RAAG, and let \(a\) and \(b\) be conjugates of standard generators. Then the following hold.
\begin{enumerate}
\item\label{Item:CommonPowersExtensionVertices}
If \(a\) and \(b\) have nonzero powers in common, then \(a=b\). Equivalently, they determine the same vertex of \(\Gamma^e\).
\item\label{Item:CommutingPowersExtensionVertices}
If some nonzero powers of \(a\) and \(b\) commute, then \(a\) and \(b\) commute. Equivalently, if \([a^p,b^q]=1\) for some nonzero integers \(p,q\), then \([a,b]=1\).
\end{enumerate}
\end{lemma}

\begin{proof}
\noindent\eqref{Item:CommonPowersExtensionVertices}
Write \(a=wv_iw^{-1}\) and \(b=w'v_j(w')^{-1}\) for standard generators \(v_i,v_j\in \calV(\Gamma)\).
Suppose that \(a^p=b^q\) for some nonzero integers \(p,q\). After abelianizing \(\bbA(\Gamma)\), we obtain \(p[v_i]=q[v_j]\). Therefore \(i=j\) and \(p=q\). Replacing both exponents by their absolute values if necessary, we obtain \(a^{|p|}=b^{|p|}\). Since RAAGs have the unique root property, \(a=b\).

\smallskip

\noindent\eqref{Item:CommutingPowersExtensionVertices}
After replacing \(p\) and \(q\) by their absolute values, we may assume that \(p,q>0\). Since \([a^p,b^q]=1\), we have \((b^qab^{-q})^p=b^qa^pb^{-q}=a^p\).
The unique root property gives \(b^qab^{-q}=a\), so \(a\) commutes with \(b^q\). Hence \((aba^{-1})^q=ab^qa^{-1}=b^q\).
Applying the unique root property again gives \(aba^{-1}=b\). Thus \([a,b]=1\).
\end{proof}

The relevance of the extension graph to RAAG embeddings is given by the following fundamental results of Kim--Koberda.

\begin{theorem}[{\cite[Theorems~1.3 and~1.4]{KK13}}]\label{Thm:KK RAAG emb Thm}
Let \(\Lambda\) and \(\Gamma\) be finite simplicial graphs.
If \(\Lambda\le \Gamma^e\), then \(\bbA(\Lambda)\le \bbA(\Gamma)\). Conversely, if \(\bbA(\Lambda)\le \bbA(\Gamma)\), then \(\Lambda\le (\Gamma^e)_k\).
\end{theorem}

Although the clique-graph condition cannot be removed in general \cite{CDK13,LL18}, in certain cases the obstruction can be strengthened to one in the extension graph itself.

\begin{theorem}\label{Thm:EquivalenceforEmbedding}
Let \(\Lambda\) and \(\Gamma\) be finite simplicial graphs. Then
\[\bbA(\Lambda)\le \bbA(\Gamma)\qquad\Longleftrightarrow\qquad\Lambda\le \Gamma^e\]
provided that one of the following holds:
\begin{enumerate}
\item {\cite[Corollary~1.9 and Theorem~1.11]{KK13}} \(\Lambda\) is a forest or \(\Gamma\) is triangle-free;
\item {\cite{Kat18}} \(\Lambda\) is the complement of a \emph{linear forest}, a finite disjoint union of paths, including isolated vertices.
\end{enumerate}
\end{theorem}

We conclude this subsection with an elementary observation about passing to powers of standard generators for RAAG embeddings. 

\begin{lemma}\label{Lem:PowerEmbeddingRAAG}
Let \(\Gamma\) be a finite graph, and let \(n_v\) be a positive integer for each \(v\in \calV(\Gamma)\). Then the assignment \(v\mapsto v^{n_v}\) extends to an injective homomorphism \(\bbA(\Gamma)\rightarrow\bbA(\Gamma)\), which is a quasi-isometric embedding.

Consequently, if \(\phi\colon\bbA(\Gamma)\rightarrow G\) is injective and \(H\le G\) has finite index, then there exist positive integers \(n_v\) such that \(v\mapsto\phi(v)^{n_v}\) extends to an injective homomorphism \(\phi':\bbA(\Gamma)\rightarrow H\). If \(G\) is finitely generated and \(\phi\) is a quasi-isometric embedding, then \(\phi'\) is also a quasi-isometric embedding.
\end{lemma}
\begin{proof}
The assertions are immediate if \(\calV(\Gamma)=\emptyset\), so assume otherwise. Equip \(\bbA(\Gamma)\) with the standard generating set \(\calV(\Gamma)\). Let \(g=v_1^{m_1}\cdots v_k^{m_k}\) be a reduced syllable expression for \(g\in\bbA(\Gamma)\), where \(m_i\neq 0\) for every \(i\). By the normal form for graph products \cite{Green90}, replacing each syllable by a nonzero power preserves reducedness. Hence, for the homomorphism \(\mu\colon\bbA(\Gamma)\rightarrow\bbA(\Gamma)\) defined by \(\mu(v)=v^{n_v}\) for \(v\in \calV(\Gamma)\), \(\mu(g)=v_1^{n_{v_1}m_1}\cdots v_k^{n_{v_k}m_k}\) is again reduced, and thus \(|\mu(g)|=\sum_{i=1}^k n_{v_i}|m_i|\). Since \(|g|=\sum_{i=1}^k |m_i|\), we obtain 
\[\left(\min_{v\in \calV(\Gamma)} n_v\right)|g|\leq|\mu(g)|\leq \left(\max_{v\in \calV(\Gamma)} n_v\right)|g|.\] The lower bound gives injectivity, and the two bounds show that \(\mu\) is a quasi-isometric embedding.

The remaining assertions follow immediately by composing with the given homomorphism \(\phi\), together with the facts that suitable positive powers of finitely many elements lie in any prescribed finite-index subgroup and that the inclusion of a finite-index subgroup of a finitely generated group is a quasi-isometry.
\end{proof}

\subsection{Rich-family HHG structures on compact special groups}
\label{Subsection:FactorSystems}

We briefly recall the rich-family construction of hierarchically hyperbolic structures on compact special groups. We refer to \cite{BH} for background on \(\CAT(0)\) cube complexes.

Let \(X\) be a \(\CAT(0)\) cube complex, which is not a single vertex. Recall that two convex subcomplexes \(Y_1,Y_2\subseteq X\) are called \emph{parallel} if they are crossed by the same set of hyperplanes. 
A \emph{factor system} on \(X\) is a collection of convex subcomplexes satisfying certain projection and local-finiteness conditions; see \cite[Definition~8.1]{BHS17I}. The associated HHS index set consists of the parallelism classes of factor-system elements other than single-vertex subcomplexes, which are omitted to keep nesting and orthogonality mutually exclusive; see \cite[Remark~13.2]{BHS17I}.
We write \([F]\) for the parallelism class of a convex subcomplex \(F\subseteq X\). We will only use factor systems arising from the following construction.

Let \(\Gamma\) be a finite simplicial graph. A collection \(\mathcal R\) of full subgraphs of \(\Gamma\) is called a \emph{rich family} if it contains \(\Gamma\) and the link of every vertex of \(\Gamma\), and is closed under intersections. By \cite[Definition~8.2 and Proposition~8.3]{BHS17I}, such a family determines an \(\bbA(\Gamma)\)-invariant factor system on the universal cover of the Salvetti complex \(S_\Gamma\), whose elements are lifts of the sub-Salvetti complexes \(S_\Lambda\) with \(\Lambda\in\mathcal R\).

More generally, if \(X\) is a compact special cube complex, then a local isometry \(X\to S_\Gamma\) lifts to a convex embedding \(\widetilde X\to\widetilde S_\Gamma\). Restricting a factor system arising from a rich family to \(\widetilde X\) gives a \(\pi_1(X)\)-invariant factor system and hence an HHG structure on \(\pi_1(X)\); see \cite[Lemma~8.5 and Corollary~8.9]{BHS17I}. We refer to the resulting HHG structures as \emph{rich-family HHG structures}. For these structures, we work in the ambient HHS \(\widetilde X\). 
For each factor-system element \(F\subseteq\widetilde X\) representing a domain \([F]\in\frS\), the convex subcomplexes parallel to \(F\) are geometric models of the standard nesting factor \(\mathbf F_{[F]}\); see \cite[Remark~13.5]{BHS17I}.

\section{Structural assumptions for algebraic orthogonality}\label{Section:Orthogonality}

In this section, we fix an HHG \((G,\frS)\) and introduce the structural assumptions used in our RAAG embedding arguments. We first define the orthogonal stabilizers \(G_U\) and explain their relation to rigidly fully supported axial elements. We then define the classes \(\Xi\) and \(\Omega\), prove the inclusion \(\Xi\subseteq\Omega\), and discuss the motivating examples. Further details of the comparison with the Abbott--Behrstock framework are given in \Cref{Appendix:ComparisonWithAB}.

\subsection{Orthogonal stabilizers and rigid full support}\label{Subsection:FUStabilizers}

We define the stabilizers directly from the prescribed action on the coordinate spaces. Only unbounded domains orthogonal to the given domain enter the definition.

\begin{definition}[Orthogonal stabilizers]\label{Def:OurGUandFU}
For \(U\in\frS\), define its \emph{orthogonal stabilizer} by
\[G_U=\{g\in\Stab_G(U):gV=V,\ g^\diamond(V)=\id_{\calC V}\text{ for every \emph{unbounded} }V\bot U\}.\]
Equivalently, \(G_U\) is the kernel of the action of \(\Stab_G(U)\) on the labelled disjoint union
\[\bigsqcup_{\substack{V\bot U\\\diam(\calC V)=\infty}}\bigl(\{V\}\times\calC V\bigr),\qquad g(V,p)=(gV,g^\diamond(V)p).\]

We say that \((G,\frS)\) satisfies the \emph{\(\mathbf F_U\) stabilizers property} if, for every \(U\in\frS\), the natural quasi-action of \(G_U\) on the standard nesting factor \(\mathbf F_U\) is proper and cobounded and induces an HHG structure \((G_U,\frS_U)\) with the restricted domain and coordinate actions.
In particular, this property requires \(G_U\) to have finitely many orbits in \(\frS_U\).
\end{definition}

The induced quasi-action above is determined, up to uniformly bounded distance in \(\mathbf F_U\), by the prescribed domain and coordinate actions. It is the restriction to \(G_U\) of the standard quasi-action of \(\Stab_G(U)\) on \(\mathbf F_U\); see \cite[Remark~1.14]{DHS17}.

We retain the term `\(\mathbf F_U\) stabilizers property' used by Abbott--Behrstock, although both the subgroups \(G_U\) and the stabilizers condition used here differ from those in \cite{AB23}; see \Cref{Rem:ComparisonWithAB}.
The subgroups \(G_U\) are determined by the same data as in \Cref{Rem:SupportAndAmbientHHS} and require no additional choice of geometric fibers.

The definition gives \(hG_Uh^{-1}=G_{hU}\) for every \(h\in G\). 
The following lemma relates \(G_U\)-orbits to geometric \(\mathbf F_U\)-fibers.

\begin{lemma}\label{Lem:OrthogonalStabilizerOrbit}
Let \((G,\frS)\) be an HHG with ambient HHS \((\mathcal X,\frS)\). For \(U\in\frS\), let \(F\subseteq\mathcal X\) be a geometric \(\mathbf F_U\)-fiber and fix \(x\in F\).
Then \(G_Ux\) lies in a bounded neighbourhood of \(F\), with a bound independent of the element of \(G_U\). If \((G,\frS)\) satisfies the \(\mathbf F_U\) stabilizers property, then \(G_Ux\) and \(F\) are at finite Hausdorff distance.
\end{lemma}
\begin{proof}
Let \(\phi_U\colon\mathbf F_U\times\mathbf E_U\to\mathcal X\) be the coarse realization map for the standard product region \(\mathbf P_U\).
Choose \(\mathbf e\in\mathbf E_U\) such that \(F\) is at uniformly bounded Hausdorff distance from \(\phi_U(\mathbf F_U\times\{\mathbf e\})\).
For \(g\in G_U\), coarse equivariance implies that the coordinates of \(gx\) on unbounded domains orthogonal to \(U\) remain uniformly close to those of \(\mathbf e\).
Bounded coordinate spaces have a uniform diameter bound, so the same conclusion holds for every orthogonal coordinate. On domains transverse to or properly containing \(U\), the coordinates remain uniformly close to the corresponding relative projections of \(U\), since \(gU=U\). Realize the nested coordinates of \(gx\) together with \(\mathbf e\) in the product region. The resulting point has uniformly bounded coordinate distance from \(gx\) in every domain, so uniqueness places \(gx\) in a bounded neighbourhood of \(F\).

Under the stabilizers property, the induced quasi-action of \(G_U\) on \(\mathbf F_U\) is cobounded. Thus every point of \(F\) has nested coordinates uniformly close to those of some \(gx\), with \(g\in G_U\). The preceding control of the other coordinates and uniqueness prove the reverse coarse containment.
\end{proof}

\begin{lemma}\label{Lem:RigidSupportCharacterization}
An axial element \(g\in G\) is rigidly fully supported on an unbounded domain \(U\) if and only if \(g\in G_U\) and \(U\in\B(g)\).
\end{lemma}
\begin{proof}
If \(g\) is rigidly fully supported on \(U\), then \(\B(g)=\{U\}\), so \(gU=U\), and the remaining conditions defining \(G_U\) hold by \Cref{Def:FSE}. Conversely, suppose that \(g\in G_U\) and \(U\in\B(g)\). Every other domain in \(\B(g)\) would be unbounded and orthogonal to \(U\), contrary to the trivial coordinate action of \(g\). Hence \(\B(g)=\{U\}\), and \(g\) is rigidly fully supported on \(U\).
\end{proof}

In particular, rigid full support is preserved by positive powers and conjugation. The following lemma is the main reason for requiring \(G_U\) to carry the induced HHG structure.

\begin{lemma}\label{Lem:RigidlyFullySupportedElements}
Suppose that an HHG \((G,\frS)\) satisfies the \(\bfF_U\) stabilizers property.
Let \(U\in\frS\) be an unbounded domain. Then the following hold.
\begin{enumerate}
\item\label{Item:ExistenceOfAxial} The subgroup \(G_U\) contains an axial element of \(G\) rigidly fully supported on \(U\). 

\item\label{Item:TwoRFS} Let \(f,g\in G\) be axial elements rigidly fully supported on \(U\). Then either \(f\) and \(g\) have nonzero powers in common, or the pair \(\{f,g\}\) is geometrically irredundant.

\item\label{Item:Quasiline} If \(\calC U\) is a quasi-line, then any two axial elements rigidly fully supported on \(U\) have nonzero powers in common.

\item\label{Item:Nonquasiline} If \(\calC U\) is not a quasi-line, then \(G_U\) contains an infinite collection of axial elements rigidly fully supported on \(U\) such that every finite subcollection is geometrically irredundant.
\end{enumerate}
\end{lemma}

\begin{proof}
\noindent\eqref{Item:ExistenceOfAxial}
By the \(\mathbf F_U\) stabilizers property, \((G_U,\frS_U)\) is an HHG whose maximal domain is \(U\). Since \(\calC U\) is unbounded, \Cref{Lem:ExistenceOfAxialElement} gives an element \(g\in G_U\) acting loxodromically on \(\calC U\).
Hence \(U\in\B_G(g)\). In particular, \(g\) has infinite order, and is therefore axial in \(G\) by \Cref{Prop:AxialEllipticDichotomy}. By \Cref{Lem:RigidSupportCharacterization}, the element \(g\) is rigidly fully supported on \(U\).

\smallskip

For the rest of the proof, we use the fact that the action \(G_U\curvearrowright\calC U\) is acylindrical, by \cite[Theorem~14.3]{BHS17I} applied to the HHG \((G_U,\frS_U)\), whose maximal domain is \(U\).

\smallskip

\noindent\eqref{Item:TwoRFS}
Let \(f,g\in G\) be axial elements rigidly fully supported on \(U\). By \Cref{Lem:RigidSupportCharacterization}, both elements belong to \(G_U\). They have bigset \(\{U\}\) and act loxodromically on \(\calC U\). If their limit sets in \(\partial\calC U\) are disjoint, then \(\{f,g\}\) is geometrically irredundant. 
Otherwise, \(\langle f,g\rangle\) fixes a point of \(\partial\calC U\). Its action on \(\calC U\) is acylindrical and contains a loxodromic element, so \cite[Theorem~1.1]{Osi16} implies that \(\langle f,g\rangle\) is virtually cyclic. Hence \(f\) and \(g\) have nonzero powers in common.

\smallskip

\noindent\eqref{Item:Quasiline}
Suppose that \(\calC U\) is a quasi-line. Any two loxodromic elements of \(G_U\curvearrowright\calC U\) have the same limit set in \(\partial\calC U\). Hence, by Item~\eqref{Item:TwoRFS}, any two axial elements rigidly fully supported on \(U\) have nonzero powers in common.

\smallskip

\noindent\eqref{Item:Nonquasiline}
Suppose that \(\calC U\) is not a quasi-line. Let \(\pi_U^{\bfF}\colon\bfF_U\rightarrow 2^{\calC U}\) denote the projection associated to the induced HHS structure \((\bfF_U,\frS_U)\). Since \(\pi_U^{\bfF}\) is coarsely surjective and uniformly coarsely \(G_U\)-equivariant, and \(G_U\curvearrowright\bfF_U\) is cobounded, the action \(G_U\curvearrowright\calC U\) is cobounded. By Item~\eqref{Item:ExistenceOfAxial}, this action contains a loxodromic element. If it were elementary, then coboundedness would imply that \(\calC U\) is a quasi-line, contrary to assumption. Thus the action is non-elementary. Since it is also acylindrical, \cite[Theorem~1.1]{Osi16} gives infinitely many pairwise independent loxodromic elements of \(G_U\). Each such element \(a\) satisfies \(U\in\B_G(a)\), and hence is an axial element of \(G\) rigidly fully supported on \(U\). Since these loxodromic elements are pairwise independent, their limit sets in \(\partial\calC U\) are pairwise disjoint. Therefore, every finite subcollection is geometrically irredundant.
\end{proof}

\begin{lemma}[Virtual commutation of orthogonal stabilizers]\label{Lem:VirtualCommutation}
Suppose that \((G,\frS)\) satisfies the \(\mathbf F_U\) stabilizers property, and let \(U,V\) be orthogonal unbounded domains. Then \(G_U\cap G_V\) is finite, and there exist finite-index subgroups \(G_U^0\le G_U\) and \(G_V^0\le G_V\) such that \([G_U^0,G_V^0]=1\). In particular, for every \(a\in G_U\) and \(b\in G_V\), some positive power of \(a\) commutes with some positive power of \(b\).
\end{lemma}
\begin{proof}
Every element of \(G_U\cap G_V\) fixes each unbounded domain nested into \(U\) and acts trivially on its coordinate space, since every such domain is orthogonal to \(V\). Bounded coordinate spaces have a uniform diameter bound. Uniqueness in \((\mathbf F_U,\frS_U)\) therefore bounds a \(G_U\cap G_V\)-orbit, and properness of the induced \(G_U\)-quasi-action implies that \(G_U\cap G_V\) is finite.

We now use the finiteness of \(G_U\cap G_V\) to obtain finite conjugacy orbits and hence finite-index centralizers. Every element of \(G_U\) fixes \(V\), so conjugacy equivariance gives \(gG_Vg^{-1}=G_V\) for \(g\in G_U\). Similarly, \(G_V\) normalizes \(G_U\). Hence \([G_U,G_V]\le G_U\cap G_V\).
Consequently, for each \(b\in G_V\), its conjugacy orbit under \(G_U\) is contained in the finite set \((G_U\cap G_V)b\). Thus \(C_{G_U}(b)\) has finite index in \(G_U\).

The induced HHG structure on \(G_V\) ensures that \(G_V\) is finitely generated. Intersecting these centralizers over a finite generating set therefore gives a finite-index subgroup \(G_U^0=C_{G_U}(G_V)\le G_U\).
By symmetry, \(G_V^0=C_{G_V}(G_U)\) has finite index in \(G_V\), and these subgroups centralize each other. The final assertion follows by taking positive powers lying in these finite-index subgroups.
\end{proof}

For rich-family HHG structures on compact special groups, the orthogonal stabilizers contain the geometric factor stabilizers with finite index.

\begin{lemma}[Stabilizers for rich-family structures]\label{Lem:RichFamilyStabilizers}
Let \(G=\pi_1(X)\), where \(X\) is a compact special cube complex. Any two parallel convex subcomplexes \(F,F'\) of \(\widetilde X\) have the same \(G\)-stabilizer. If \(\widetilde X\) is equipped with the factor system arising from a rich family and \(U=[F]\in\frS\), then
\[\Stab_G(F)\le G_U\quad\text{and}\quad [G_U:\Stab_G(F)]<\infty.\]
\end{lemma}
\begin{proof}
Use the oriented edge labelling on \(\widetilde X\) induced by a local isometry \(X\to S_\Gamma\). This labelling is preserved by \(G\), and a combinatorial path is determined by its initial vertex and its label word. If \(F,F'\) are parallel, then \cite[Lemma~2.4]{BHS17I} provides a cubical isometric embedding \(F\times[0,a]\to\widetilde X\) with boundary fibers \(F,F'\); let \(\tau\colon F\to F'\) be the resulting parallelism map. For \(x\in F^{(0)}\), let \(q_x\) be the path corresponding to \(\{x\}\times[0,a]\). All these paths have the same oriented label word. Thus, for \(g\in\Stab_G(F)\), the paths \(gq_x\) and \(q_{gx}\) have the same initial vertex and label word, so \(g\tau(x)=\tau(gx)\). It follows that \(gF'=F'\), and symmetry proves the first assertion.

Let \(P_F=F\times E_F\subseteq\widetilde X\) be the cubical product region from \cite[Lemma~2.4]{BHS17I}. Its factors are geometric models of \(\mathbf F_U\) and \(\mathbf E_U\), respectively; see \cite[Remark~13.5]{BHS17I}. The equivariance of the parallelism maps just proved shows that every \(g\in\Stab_G(F)\) acts on \(P_F\) as \((g|_F,\id_{E_F})\). If \(V\bot U\), choose a representative of \(V\) in \(\{x\}\times E_F\) for \(x\in F^{(0)}\); see \cite[Proposition~10.1 and Definition~10.2]{BHS17I}. The element \(g\) carries this representative to a parallel copy by the canonical parallelism map, so \(gV=V\). Under the canonical identification of their factored contact graphs, this map fixes both hyperplane vertices and cone vertices, and hence \(g^\diamond(V)=\id_{\calC V}\); see \cite[Lemma~8.17]{BHS17I}. Thus \(\Stab_G(F)\le G_U\).

The action \(\Stab_G(F)\curvearrowright F\) is cocompact; see \cite[Lemma~2.3 and Remark~2.4]{HS20}. By \Cref{Lem:OrthogonalStabilizerOrbit}, the orbit \(G_Ux\), for \(x\in F\), lies in a bounded neighbourhood of \(F\), and hence of \(\Stab_G(F)x\). Properness of the ambient action gives a finite set \(A\subseteq G\) such that \(G_U\subseteq\Stab_G(F)A\). This proves the finite-index assertion.
\end{proof}

\subsection{The classes \texorpdfstring{\(\Xi\)}{Xi} and \texorpdfstring{\(\Omega\)}{Omega}}\label{Subsection:Xi}

The permanence properties stated below are proved in \Cref{Appendix:OmegaPermanence}. 

\subsubsection{The class \texorpdfstring{\(\Xi\)}{Xi} motivated by compact special groups}

Inspired by the framework of Abbott--Behrstock \cite{AB23}, we supplement the \(\mathbf F_U\) stabilizers property with the following decomposition condition. Virtual commutation is already supplied by \Cref{Lem:VirtualCommutation}.

\begin{definition}\label{Def:OrthogonalDecomposition}
An HHG \((G,\frS)\) satisfies the \emph{orthogonal decomposition property} if, for every infinite-order element \(g\in G\) with \(\B(g)=\{U_1,\dots,U_k\}\), there exist a positive integer \(n\) and pairwise commuting axial elements \(h_1,\dots,h_k\in G\) such that \(g^n=h_1\cdots h_k\), where \(h_i\) is rigidly fully supported on \(U_i\) for every \(i\).

We denote by \(\Xi\) the class of HHG structures satisfying the \(\mathbf F_U\) stabilizers property and the orthogonal decomposition property.
\end{definition}

The following is a direct consequence of \Cref{Prop:SupportedAxialRAAGs} together with the structural assumptions defining \(\Xi\).

\begin{lemma}\label{lem:commuting factors}
Let \((G,\frS)\in\Xi\), and let \(f_1,f_2\in G\) be axial elements rigidly fully supported on unbounded domains. If \(f_1\) and \(f_2\) commute, then either they have nonzero powers in common or their supporting domains are orthogonal.
\end{lemma}
\begin{proof}
Let \(f_i\) be rigidly fully supported on \(U_i\) for \(i=1,2\), and suppose that \(f_1\) and \(f_2\) commute and have no nonzero powers in common.
If \(U_1\neq U_2\), then \(\B(f_1)=\{U_1\}\neq\{U_2\}=\B(f_2)\), so \(\{f_1,f_2\}\) is geometrically irredundant. If \(U_1=U_2\), the same conclusion follows from \Cref{Lem:RigidlyFullySupportedElements}\eqref{Item:TwoRFS}. Thus \(\{f_1,f_2\}\) is geometrically irredundant in either case.

If \(U_1\not\bot U_2\), then the commutation hypothesis in \Cref{Prop:SupportedAxialRAAGs} is vacuous. Since rigid full support implies full support, that proposition implies that sufficiently large powers of \(f_1\) and \(f_2\) generate a nonabelian free group, contradicting the assumption that they commute. Therefore \(U_1\bot U_2\).
\end{proof}

\begin{remark}[Comparison with Abbott--Behrstock]\label{Rem:ComparisonWithAB}
The class \(\Xi\) above is inspired by, but is not identical to, the class denoted by \(\Xi\) in \cite{AB23}. Their stabilizers, denoted here by \(G_U^{\mathrm{AB}}\), are defined using geometric realizations of the standard factors; see the discussion following \cite[Definition~1.17]{AB23} and \cite[Definition~3.1]{AB23}. Our subgroups are defined directly from the prescribed action on the unbounded orthogonal coordinate spaces and require no additional choice of geometric fibers. We explicitly require the natural induced quasi-action on the nesting factor to yield an induced HHG structure on \(G_U\).

Under our stabilizers property, \Cref{Lem:ABFiniteIndex} gives
\[[G_U^{\mathrm{AB}}:G_U^{\mathrm{AB}}\cap G_U]<\infty.\]
If the stabilizers property of \cite[Definition~3.3]{AB23} also holds, then \(G_U^{\mathrm{AB}}\) and \(G_U\) are commensurable. Consequently, an HHG structure in the class \(\Xi\) of \cite{AB23} which also satisfies our stabilizers property belongs to our class \(\Xi\): suitable common positive powers of its commuting decomposition factors lie in our subgroups and are rigidly fully supported. The details are given in \Cref{Rem:ComparisonWithABConti}. Our formulation does not require clean containers or a separate exact commutation axiom.

Requiring an orbit of \(G_U^{\mathrm{AB}}\) to be coarsely equal to a geometric \(\mathbf F_U\)-fiber does not by itself imply that \(G_U^{\mathrm{AB}}\) has finitely many orbits in \(\frS_U\); see \Cref{Ex:BoundedDomainPathology}.
\end{remark}

\begin{proposition}\label{Prop:Xiclass}
The following HHG structures belong to \(\Xi\):
\begin{enumerate}
\item the one-domain HHG structures on hyperbolic groups;
\item rich-family HHG structures on compact special groups;
\item\label{Item:HyperbolicRelto}
the relatively hyperbolic HHG structures whose peripheral structures belong to \(\Xi\);
\item\label{Item:XiDirectProducts}
the rectangular product structures of \Cref{Lem:RectangularProductStructure} on \(G_1\times G_2\), where \((G_1,\frS_1)\) and \((G_2,\frS_2)\) belong to \(\Xi\).
\end{enumerate}
\end{proposition}

\begin{proof}
\noindent\emph{(1) Hyperbolic groups.}
Let \(G\) be a hyperbolic group equipped with the one-domain HHG structure, whose unique domain is \(S\). Then \(G_S=G\) and \(\frS_S=\{S\}\), so the \(\mathbf F_U\) stabilizers property holds. Every infinite-order element \(g\in G\) has \(\B(g)=\{S\}\), and, since no domain is orthogonal to \(S\), it is rigidly fully supported on \(S\). Thus \(g\) itself gives the required one-factor decomposition.

\smallskip

\noindent\emph{(2) Compact special groups.}
Let \(X\) be a compact special cube complex, and equip \(G=\pi_1(X)\) with a rich-family HHG structure. For \(U=[F]\in\frS\), put \(H_U=\Stab_G(F)\). The subcomplex \(F\), with its induced HHS structure, is a geometric model of \((\mathbf F_U,\frS_U)\). Its induced factor system is \(H_U\)-invariant and locally finite, and the action \(H_U\curvearrowright F\) is proper and cocompact; see \cite[Lemma~2.3 and Remark~2.4]{HS20}. Thus the induced action of \(H_U\) defines an HHG structure. By \Cref{Lem:RichFamilyStabilizers}, \(H_U\) has finite index in \(G_U\). The natural \(G_U\)-quasi-action on the nesting factor is therefore also proper and cobounded, and its action on \(\frS_U\) is cofinite. Hence our \(\mathbf F_U\) stabilizers property holds.

For the cubical geometric fibers, \Cref{Lem:RichFamilyStabilizers} gives \(G_U^{\mathrm{AB}}=H_U\le G_U\). Let \(g\in G\) have infinite order. The orthogonal decomposition and commutative properties verified in \cite[Example~3.2 and Proposition~3.10(2)]{AB23} give, after taking a positive power, a decomposition \(g^n=h_1\cdots h_k\) into pairwise commuting elements \(h_i\in H_{U_i}\), where \(\B(g)=\{U_1,\dots,U_k\}\), and \(g^n\) acts loxodromically on every \(\calC U_i\). For \(j\neq i\), the element \(h_j\in G_{U_j}\) acts trivially on \(\calC U_i\). Hence \(h_i\) acts there as \(g^n\) and is loxodromic. By \Cref{Lem:RigidSupportCharacterization}, each \(h_i\) is rigidly fully supported on \(U_i\). This proves orthogonal decomposition and hence membership in \(\Xi\).

\smallskip

\noindent\emph{(3) Relatively hyperbolic groups.} The proof is given in \Cref{Appendix:OmegaRelativelyHyperbolic}.

\smallskip

\noindent\emph{(4) Direct products.} We use a rectangular variant of the construction in \cite[Proposition~8.27]{BHS19II}. The choice of auxiliary bounded domains for this construction and the proof are given in \Cref{Appendix:OmegaDirectProducts}.
\end{proof}

\begin{remark}\label{Rem:WhyOmega}
In contrast to the examples in \Cref{Prop:Xiclass}, the standard HHG structure on \(\MCG(S)\) does not belong to \(\Xi\) when \(S\) is a connected finite-type orientable surface of complexity \(\xi(S)=3g-3+p+b\geq2\), where \(g\), \(p\), and \(b\) denote its genus, number of punctures, and number of boundary components, respectively. Here we use the standard HHG structure, including its extension to surfaces with boundary described in \cite[Theorem~9.2 and Remark~9.3]{HMS26}.

Indeed, choose a connected proper non-annular essential subsurface \(U\) of positive complexity with an essential boundary curve in the interior of \(S\).
The discussion in \cite[\S7.1]{OP25} provides an axial element \(f\) fully supported on \(U\) such that no positive power of \(f\) is rigidly fully supported: after passing to a pure power, its pseudo-Anosov restriction acts nontrivially on the curve graph of a boundary annulus of \(U\), and the same argument applies to every further positive power.
If the standard structure belonged to \(\Xi\), then its orthogonal decomposition property and \(\B(f)=\{U\}\) would imply that some positive power of \(f\) is rigidly fully supported, a contradiction.
\end{remark}

\subsubsection{The class \texorpdfstring{\(\Omega\)}{Omega} motivated by mapping class groups}\label{Subsection:Omega}

To include mapping class groups in our framework for RAAG embeddings, we formulate the required existence, decomposition, and commutation properties directly in terms of fully supported axial elements.
These properties suffice for the algebraic factorization and obstruction results, while the stronger support conditions available in \(\Xi\) ensure the additional bounded-orbit hypothesis in \Cref{Prop:SupportedAxialRAAGs} needed for the undistortion conclusion.

\begin{definition}\label{Def:Omega}
Let \((G,\frS)\) be an HHG. We say that \((G,\frS)\) belongs to \(\Omega\) if the following conditions hold.
\begin{enumerate}
\item\label{Omega:AxialDirections}
For every unbounded domain \(U\in\frS\), there exists an axial element fully supported on \(U\). Moreover, if \(\calC U\) is not a quasi-line, then there exists an infinite collection of axial elements fully supported on \(U\) such that every finite subcollection is geometrically irredundant.
\item\label{Omega:Decomposition}
For every infinite-order element \(g\in G\) with \(\B(g)=\{U_1,\dots,U_k\}\), there exist a positive integer \(n\) and pairwise commuting axial elements \(h_1,\dots,h_k\in G\) such that \(g^n=h_1\cdots h_k\), where \(h_i\) is fully supported on \(U_i\) for every \(i\).
\item\label{Omega:CommutationPackage}
Let \(f_1\) and \(f_2\) be axial elements fully supported on unbounded domains \(U_1\) and \(U_2\), respectively. Then:
\begin{enumerate}
\item\label{Omega:SameSupportDichotomy}
If \(U_1=U_2\), then either \(f_1\) and \(f_2\) have nonzero powers in common, or the pair \(\{f_1,f_2\}\) is geometrically irredundant.
\item\label{Omega:WeakCommutativity}
If \(U_1\bot U_2\), then there exists a positive integer \(N\) such that \([f_1^N,f_2^N]=1\). 
\item\label{Omega:CommutationDetection}
If \([f_1,f_2]=1\), then either \(f_1\) and \(f_2\) have nonzero powers in common, or \(U_1\bot U_2\).
\end{enumerate}
\end{enumerate}
\end{definition}

The commutation-detection condition follows from the same-support dichotomy and \Cref{Prop:SupportedAxialRAAGs}; we record it explicitly for later use. Indeed, two commuting fully supported elements with no nonzero powers in common are geometrically irredundant, by the same-support dichotomy or their distinct bigsets. If their supporting domains were not orthogonal, then \Cref{Prop:SupportedAxialRAAGs} would imply that sufficiently large powers generate a nonabelian free group, a contradiction.

\begin{proposition}\label{Prop:XiContainedInOmega}
Every HHG structure belonging to \(\Xi\) belongs to \(\Omega\).
\end{proposition}
\begin{proof}
Let \((G,\frS)\in\Xi\). If \(f\) is fully supported on \(U\), then \(\B(f)=\{U\}\), so the orthogonal decomposition property implies that some positive power of \(f\) is rigidly fully supported on \(U\). Thus, after taking common positive powers, the same-support dichotomy and the two commutation conditions in \Cref{Def:Omega} follow from \Cref{Lem:RigidlyFullySupportedElements,Lem:VirtualCommutation,lem:commuting factors}. 
Both geometric irredundancy and the property of having nonzero powers in common are unchanged by taking positive powers.

The existence condition follows from \Cref{Lem:RigidlyFullySupportedElements}, and the decomposition condition follows directly from orthogonal decomposition, since rigidly fully supported elements are fully supported.
\end{proof}

We now verify that the standard HHG structures on mapping class groups belong to \(\Omega\).
For a surface with boundary, \(\MCG(S)\) denotes the group of isotopy classes of orientation-preserving homeomorphisms fixing the boundary pointwise, with isotopies also taken relative to the boundary. We call a simple closed curve \emph{essential} if it does not bound a disk or a once-punctured disk; in particular, boundary-parallel curves are included.

\begin{proposition}\label{Prop:MCGinOmega}
Let \(S\) be a connected finite-type orientable surface, possibly with boundary, with \(\chi(S)<0\).
Then \(\MCG(S)\), equipped with its standard HHG structure, belongs to \(\Omega\).
\end{proposition}

Here we use the subsurface HHG structure of \cite[Theorem~11.1]{BHS19II}, with the extension to surfaces with boundary given in \cite[Theorem~9.2 and Remark~9.3]{HMS26}.

\begin{proof}
We use the standard Nielsen--Thurston decomposition and its interpretation in terms of full support, as discussed in \cite[\S7.1]{OP25}; see also \cite[\S6.1]{Kob12}. After passing to a positive power, every infinite-order mapping class is a product of pairwise commuting pure mapping classes with connected supports, one for each active domain. Each factor is fully supported on its domain, giving \Cref{Def:Omega}\eqref{Omega:Decomposition}. In particular, since a fully supported element \(f\) has \(\B(f)=\{U\}\), a positive power of \(f\) is pure with connected support \(U\).

The existence of independent pseudo-Anosov mapping classes on unbounded non-annular domains, together with Dehn twists on annuli, gives \Cref{Def:Omega}\eqref{Omega:AxialDirections}. For pure mapping classes with connected support, geometric irredundancy is equivalent to having no nonzero powers in common, by \cite[\S7.1]{OP25}. This gives \Cref{Def:Omega}\eqref{Omega:SameSupportDichotomy}, since both properties are unchanged by passing to positive powers. Here the bounded-orbit condition on boundary annuli excludes an additional boundary multitwist, as explained there; compare \cite[\S6.5]{Kob12}.

If two fully supported elements have orthogonal supports, then positive powers that are pure have disjoint supports and hence commute, proving \Cref{Def:Omega}\eqref{Omega:WeakCommutativity}. Commutation detection follows from the same-support dichotomy by the observation after \Cref{Def:Omega}. This completes the proof.
\end{proof}

\begin{remark}[Relations between classes of HHG structures]\label{Rem:XiandP_2}
In the notation of \cite[\S1.3]{OP25}, let \(\mathcal P_1\) denote the class of HHG structures satisfying the \emph{weak commutativity property}: axial elements fully supported on orthogonal unbounded domains have commuting positive powers.
Let \(\mathcal P_2\) denote the class in which every unbounded domain admits a fully supported axial element, and let \(\mathcal P'_2\) denote the analogous class for rigidly fully supported axial elements.
Then
\[\Xi\subseteq\mathcal P'_2\qquad\text{and}\qquad\Omega\subseteq\mathcal P_1\cap\mathcal P_2.\]
The first inclusion follows from \Cref{Lem:RigidlyFullySupportedElements}.
The second follows directly from \Cref{Def:Omega}\eqref{Omega:AxialDirections} and~\eqref{Omega:WeakCommutativity}.
The decomposition and commutation conditions defining \(\Omega\) provide the additional ingredients needed to pass from the constructive embedding results of \cite{OP25} to the factorization and obstruction results of the present paper.

Moreover, the inclusion \(\Xi\subseteq\Omega\) is strict by \Cref{Prop:MCGinOmega,Rem:WhyOmega}.
\end{remark}

We also record the following permanence properties of \(\Omega\).

\begin{proposition}[Permanence of \(\Omega\)]\label{Prop:OmegaPermanence}
The following hold.
\begin{enumerate}
\item\label{Item:OmegaDirectProducts}
Let \((G_1,\frS_1)\) and \((G_2,\frS_2)\) belong to \(\Omega\). Then the rectangular product structure on \(G_1\times G_2\) given by \Cref{Lem:RectangularProductStructure} belongs to \(\Omega\).

\item\label{Item:OmegaRelativelyHyperbolic}
Let \(G\) be hyperbolic relative to a finite collection \(\mathcal P\). Suppose that each \(P\in\mathcal P\) is equipped with an HHG structure \((P,\frS_P)\in\Omega\). Then the standard relatively hyperbolic HHG structure on \(G\) belongs to \(\Omega\).
\end{enumerate}
\end{proposition}

\begin{proof}
The two assertions are proved in \Cref{Appendix:OmegaDirectProducts,Appendix:OmegaRelativelyHyperbolic}.
\end{proof}

We finish this section by recording that both \(\Xi\) and \(\Omega\) are preserved under passage to finite-index subgroups equipped with the restricted HHG structure.

\begin{lemma}[Finite-index permanence]\label{Lem:FiniteIndexXiOmega}
Let \((G,\frS)\) be an HHG, and let \(H\le G\) be a finite-index subgroup equipped with the restricted HHG structure. Then the following hold.
\begin{enumerate}
\item\label{Item:WhenGisXi} If \((G,\frS)\in\Xi\), then \((H,\frS)\in\Xi\).
\item\label{Item:WhenGisOmega} If \((G,\frS)\in\Omega\), then \((H,\frS)\in\Omega\).
\end{enumerate}
\end{lemma}

\begin{proof}
The proof is given in \Cref{Appendix:FiniteIndex}.
\end{proof}

\section{Intermediate RAAGs}\label{Section:IntermediateRAAG}

In this section, we show that every RAAG embedding into a group virtually in \(\Omega\) (in the sense of \Cref{Def:VirtualHHGStructure}), after passing to positive powers of the images of the standard generators, factors through an intermediate RAAG subgroup generated by fully supported axial elements. When the virtual HHG structure belongs to \(\Xi\), the intermediate RAAG can also be chosen to be undistorted.

The same construction gives an enveloping RAAG for any finite collection of elements and yields Wise's power alternative. We begin by removing repetitions among the axial factors that have nonzero powers in common.

\begin{lemma}\label{Lem:FiniteCommonPowerReduction}
Let \((G,\frS)\in\Omega\). Let \(U\in\frS\) be an unbounded domain, and let \(\mathcal H\) be a finite collection of axial elements, all fully supported on \(U\). Then there exists a finite geometrically irredundant subcollection \(\mathcal A=\{\alpha_1,\dots,\alpha_r\}\subseteq\mathcal H\) such that, for every \(h\in\mathcal H\), there exist a positive integer \(p_h\), a nonzero integer \(q_h\), and some \(\alpha_j\in\mathcal A\) satisfying $h^{p_h}=\alpha_j^{q_h}$.
\end{lemma}

\begin{proof}
Choose a maximal geometrically irredundant subcollection \(\mathcal A=\{\alpha_1,\dots,\alpha_r\}\subseteq\mathcal H\).
For each \(h\in\mathcal H\setminus\mathcal A\), maximality implies that \(\{h,\alpha_j\}\) is not geometrically irredundant for some \(j\). Since both elements are fully supported on \(U\), \Cref{Def:Omega}\eqref{Omega:SameSupportDichotomy} implies that they have nonzero powers in common. Reversing both exponents if necessary gives \(h^{p_h}=\alpha_j^{q_h}\) with \(p_h>0\) and \(q_h\neq0\). For \(h\in\mathcal A\), take \(p_h=q_h=1\).
\end{proof}

\begin{theorem}[Intermediate RAAG]\label{Prop:IntermediateRAAG}
Let \(G\) be virtually in \(\Omega\), via a virtual HHG structure \((H,\frS)\), and let \(\phi\colon\bbA(\Lambda)\to G\) be an injective homomorphism. Then there exist a positive integer \(N\) and a RAAG subgroup \(M\le H\) such that the assignment \(v\mapsto\phi(v)^N\) extends to an injective homomorphism \(\bbA(\Lambda)\rightarrow M\).
The subgroup \(M\) admits a standard generating set of fully supported axial elements with respect to \((H,\frS)\), with two distinct generators adjacent in its defining graph if and only if their supporting domains are orthogonal. 
If \((H,\frS)\in\Xi\), then these standard generators may be chosen to be rigidly fully supported and \(M\) may be chosen to be undistorted in \(G\).
\end{theorem}

\begin{proof}
If \(\calV(\Lambda)=\emptyset\), take \(N=1\) and \(M=\{1\}\), so assume that \(\calV(\Lambda)\neq\emptyset\).
Since \(H\) has finite index in \(G\) and \(\calV(\Lambda)\) is finite, some positive integer \(a\) satisfies \(\phi(v)^a\in H\) for every \(v\in\calV(\Lambda)\). By \Cref{Lem:PowerEmbeddingRAAG}, the assignment \(v\mapsto\phi(v)^a\) extends to an injective homomorphism \(\bbA(\Lambda)\rightarrow H\).
Replacing \(\phi\) by this homomorphism, we may therefore assume that \(\phi\colon\bbA(\Lambda)\rightarrow H\).

For each \(v\in \calV(\Lambda)\), write \(\B(\phi(v))=\{U_{v,1},\dots,U_{v,k_v}\}\).
By \Cref{Def:Omega}\eqref{Omega:Decomposition}, there exists \(n_v>0\) such that \(\phi(v)^{n_v}=h_{v,1}\cdots h_{v,k_v}\), where the elements \(h_{v,1},\dots,h_{v,k_v}\) pairwise commute and \(h_{v,s}\) is fully supported on \(U_{v,s}\) for every \(s\).
If \((H,\frS)\in\Xi\), then we choose the commuting factors using its orthogonal decomposition property, so that each \(h_{v,s}\) is rigidly fully supported.
Since \(\calV(\Lambda)\) is finite, we may choose a common positive multiple \(n\) of the integers \(n_v\). Replacing each \(h_{v,s}\) by \(h_{v,s}^{n/n_v}\), we may assume that \(\phi(v)^n=h_{v,1}\cdots h_{v,k_v}\) for every \(v\in \calV(\Lambda)\), where each \(h_{v,s}\) remains fully supported on \(U_{v,s}\).

Let \(\mathcal H\) be the finite collection of all elements \(h_{v,s}\), as \(v\) ranges over \(\calV(\Lambda)\) and \(1\leq s\leq k_v\). For each unbounded domain \(U\) occurring as the support of an element of \(\mathcal H\), let
\[\mathcal H_U=\{h\in\mathcal H\mid h\text{ is fully supported on }U\}.\]
Applying \Cref{Lem:FiniteCommonPowerReduction} to each nonempty \(\mathcal H_U\) and taking the union of the resulting subcollections, we obtain a finite geometrically irredundant collection \(\mathcal A=\{\alpha_1,\dots,\alpha_r\}\subseteq\mathcal H\) such that, for each \(h_{v,s}\), there exist a positive integer \(p_{v,s}\), a nonzero integer \(q_{v,s}\), and an index \(j(v,s)\in\{1,\dots,r\}\) satisfying \(h_{v,s}^{p_{v,s}}=\alpha_{j(v,s)}^{q_{v,s}}\).
Indeed, elements chosen from different \(\mathcal H_U\)'s have different supporting domains, while elements chosen from the same \(\mathcal H_U\) are geometrically irredundant by \Cref{Lem:FiniteCommonPowerReduction}.

By \Cref{Def:Omega}\eqref{Omega:WeakCommutativity}, there exists a common positive integer \(k\) such that \([\alpha_i^k,\alpha_j^k]=1\) whenever their supporting domains are orthogonal. By \Cref{Prop:SupportedAxialRAAGs}, applied to \((H,\frS)\), for every sufficiently large multiple \(L\) of \(k\), the subgroup
\[M=\langle\alpha_1^L,\dots,\alpha_r^L\rangle\le H\]
is a RAAG whose defining graph is given by the orthogonality of the supporting domains; we fix such an \(L\).
If \((H,\frS)\in\Xi\), then the elements \(\alpha_i\) are rigidly fully supported by our choice of decomposition factors. For each \(i\), the subgroup generated by the elements \(\alpha_j\) whose supports are orthogonal to that of \(\alpha_i\) acts trivially on the coordinate space associated to the support of \(\alpha_i\). Thus the bounded-orbit hypothesis of \Cref{Prop:SupportedAxialRAAGs} holds with \(k_2=1\) and \(Q=0\). By the same proposition, we may choose \(L\) to be a sufficiently large multiple of \(k\) so that \(M\hookrightarrow H\) is a quasi-isometric embedding. Since \(H\) has finite index in \(G\), the inclusion \(M\hookrightarrow G\) is also a quasi-isometric embedding.

Set
\[P=L\operatorname{lcm}\{p_{v,s}\mid v\in\calV(\Lambda),\ 1\leq s\leq k_v\}.\]
Then \(L\mid P/p_{v,s}\) for every \(v,s\), and hence
\[h_{v,s}^{P}=\left(h_{v,s}^{p_{v,s}}\right)^{P/p_{v,s}}=\alpha_{j(v,s)}^{q_{v,s}P/p_{v,s}} \in M.\]
Since the elements \(h_{v,1},\dots,h_{v,k_v}\) commute, for every \(v\in \calV(\Lambda)\), we obtain
\[\phi(v)^{nP}=\left(h_{v,1}\cdots h_{v,k_v}\right)^P=h_{v,1}^P\cdots h_{v,k_v}^P \in M.\]
Thus, for the modified embedding, the assignment \(v\mapsto\phi(v)^{nP}\) defines an injective homomorphism into \(M\). Returning to the original embedding, this is precisely the assignment \(v\longmapsto\phi(v)^{anP}\).
Setting \(N=anP\) proves the theorem.
\end{proof}

\begin{corollary}[Finite collections and intermediate RAAGs]\label{Cor:FiniteSetRAAGs}
Let \(G\) be virtually in \(\Omega\), via a virtual HHG structure \((H,\frS)\), and let \(F\subseteq G\) be finite. There exist a positive integer \(N\) and a RAAG subgroup \(M\le H\) with a standard generating set of fully supported axial elements with respect to \((H,\frS)\). Two distinct generators are adjacent in its defining graph if and only if their supporting domains are orthogonal. 
For each \(g\in F\), the power \(g^N\) is a product of powers of pairwise commuting standard generators of \(M\).
If \((H,\frS)\in\Xi\), then these standard generators may be chosen to be rigidly fully supported and \(M\) may be chosen to be undistorted in \(G\).
\end{corollary}
\begin{proof}
Choose a positive integer \(a\) such that \(g^a\in H\) for every \(g\in F\), and \(g^a=1\) whenever \(g\) has finite order. If every element of \(F\) has finite order, take \(N=a\) and \(M=\{1\}\), the RAAG on the empty graph.
Otherwise, apply the decomposition, common-power reduction, and RAAG construction in the proof of \Cref{Prop:IntermediateRAAG} to the infinite-order elements of \(\{g^a\mid g\in F\}\). These steps do not use the assumption that the original elements are images of standard generators of a RAAG. They give a positive integer \(b\) and a subgroup \(M\) as above such that \(g^{ab}\in M\) for every \(g\in F\). For each infinite-order \(g\in F\), the decomposition factors of a suitable power of \(g^a\) have pairwise orthogonal supporting domains. The common-power reduction preserves these supports, so \(g^{ab}\) is a product of powers of standard generators spanning a clique. The assertion is immediate for finite-order elements, since \(g^{ab}=1\). In the \(\Xi\) case, the same construction uses rigidly fully supported factors and produces an undistorted subgroup. Set \(N=ab\).
\end{proof}

In Koberda's terminology \cite[\S1.2]{Kob12}, the subgroup \(\langle g^N\mid g\in F\rangle\) is \emph{enveloped} by \(M\): each indicated generator is a product of powers of pairwise commuting standard generators of \(M\).
The enveloped subgroup need not itself be a RAAG; see \cite[Proposition~5.5]{Kob12}. The common enveloping RAAG also yields Wise's power alternative, which is established for pairs of mapping classes in \cite[Theorem~B]{McCarthy85PA}.

\begin{corollary}[Wise's power alternative]\label{Cor:PowerAlternative}
Let \(G\) be virtually in \(\Omega\), via a virtual HHG structure \((H,\frS)\), and let \(F\subseteq G\) be finite. There is a positive integer \(N\) such that, for every \(g,h\in F\), either \(g^N\) and \(h^N\) commute or they freely generate a subgroup of rank two. If \((H,\frS)\in\Xi\), then \(\langle g^N,h^N\rangle\) is undistorted in \(G\) for every \(g,h\in F\).
\end{corollary}
\begin{proof}
Choose \(N\) and \(M\) as in \Cref{Cor:FiniteSetRAAGs}. For each \(g,h\in F\), the elements \(g^N,h^N\) lie in the RAAG \(M\), so by Baudisch's theorem \cite{Baudisch81PA} they either commute or freely generate a subgroup of rank two. In either case their subgroup is undistorted in \(M\) by \cite[Corollary~1.5]{Carr15PA}. If \((H,\frS)\in\Xi\), then we can choose \(M\) undistorted in \(G\), and the conclusion follows.
\end{proof}

\begin{remark}
Baudisch's theorem gives exponent one for RAAGs and hence a uniform exponent for virtually compact special groups, by passing to a finite-index subgroup that embeds in a RAAG. Hyperbolic groups and their central extensions by \(\mathbb Z\) also admit an exponent depending only on the group \cite[Theorem~6.5 and Lemma~6.6]{HMSPower26}. A group hyperbolic relative to a finite collection of subgroups also satisfies the uniform alternative if each of those subgroups has bounded torsion and satisfies the uniform alternative \cite[Theorem~6.9]{HMSPower26}. For HHGs, \Cref{Cor:PowerAlternative} gives a sufficient structural condition for the ordinary power alternative, as sought in \cite[Question~6.21]{HMSPower26}; the exponent may depend on \(F\).
\end{remark}

We return to \Cref{Prop:IntermediateRAAG} and write \(M\cong\bbA(\Gamma)\), where \(\Gamma\) is the defining graph corresponding to the standard generating set supplied by the theorem.
Let \(\rho_1\colon\bbA(\Lambda)\rightarrow\bbA(\Lambda)\) be the embedding given by \(\rho_1(v)=v^N\) for \(v\in\calV(\Lambda)\), and let \(\rho_2\colon\bbA(\Gamma)\rightarrow H\) identify \(\bbA(\Gamma)\) with \(M\le H\). Write \(\iota\colon H\hookrightarrow G\) for the inclusion.
Since \((\phi\circ\rho_1)(\bbA(\Lambda))\le M\), there is an injective homomorphism \(\bar\phi\colon\bbA(\Lambda)\rightarrow\bbA(\Gamma)\) such that \(\iota\circ\rho_2\circ\bar\phi=\phi\circ\rho_1\).
Thus we obtain the commutative diagram
\[\begin{tikzcd}[row sep=0pc]
\bbA(\Lambda)\ar[r, "\bar\phi"] \ar[rrr, bend right=18, "\phi\circ\rho_1"'] & \bbA(\Gamma) \ar[r, "\rho_2"] & H \ar[r, hook, "\iota"] & G.
\end{tikzcd}\]
The map \(\rho_1\) is a quasi-isometric embedding, and \(\iota\) is a quasi-isometry since \(H\) has finite index in \(G\). If \((H,\frS)\in\Xi\), then \(\rho_2\) may also be chosen to be a quasi-isometric embedding, and hence so is the composite \(\iota\circ\rho_2\colon\bbA(\Gamma)\rightarrow G\).

\begin{remark}[Undistortion in the factorization]\label{Rem:DistortedPoweredEmbedding}
Even when \((H,\frS)\in\Xi\), one cannot in general choose both \(\bar\phi\) and the inclusion of the intermediate RAAG into \(G\) to be quasi-isometric embeddings. To see this, let \(P_4\) be the path \(a-b-c-d\), and equip \(G=\bbA(P_4)\) with its standard rich-family HHG structure, which belongs to \(\Xi\) by \Cref{Prop:Xiclass}. The kernel of the homomorphism \(G\to\mathbb Z\) sending each standard generator to \(1\) is the free group of rank three with basis
\[x=ab^{-1},\qquad y=bc^{-1},\qquad z=cd^{-1}.\]
Since \(P_4\) is not a join, this Bestvina--Brady subgroup is quadratically distorted in \(G\); see \cite[Theorems~1.1 and~2.6]{Tran17Distortion}.

Let \(\phi\colon F(x,y,z)\hookrightarrow G\) denote this inclusion. For \(N>0\), let \(D_N\colon G\to G\) send each standard generator to its \(N\)-th power. Each of \(x,y,z\) is a product of powers of two commuting standard generators, so \(D_N(x)=x^N\), \(D_N(y)=y^N\), and \(D_N(z)=z^N\). Thus the powered embedding in \Cref{Prop:IntermediateRAAG} is \(D_N\circ\phi\). Since \(D_N\) is \(N\)-Lipschitz for the standard word metric, a sequence witnessing distortion of \(\phi\) also witnesses distortion of \(D_N\circ\phi\) for every fixed \(N\). Consequently the two maps in the factorization cannot both be quasi-isometric embeddings. This example also shows that the undistortion conclusion of \Cref{Cor:PowerAlternative} can fail for three elements even when their powers freely generate.
\end{remark}

\section{The expanded core graph}\label{Section:Criterion}

In this section, we introduce the expanded core graph and describe it for the motivating examples. We then show that the extension graphs of the intermediate RAAGs constructed in \Cref{Section:IntermediateRAAG} embed into the expanded core graph as induced subgraphs. This provides the link between the algebraic factorization and the combinatorial obstructions in \Cref{Section:Obstructions}.

\subsection{Definition and basic properties}

\begin{definition}[Expanded core graph]\label{Def:ExpandedCoreGraph}
Let \((G,\frS)\) be an HHG. We define \(\frS^\infty=\{U\in\frS\mid \diam(\calC U)=\infty\}\) and 
\[\frG=\left\{U\in\frS^\infty\mid \text{there is no unbounded domain \(V\in\frS\) with \(V\sqsubsetneq U\)}\right\}.\]
The elements of \(\frG\) are called \emph{minimal unbounded domains}, and the \emph{core graph} of \((G,\frS)\), denoted \(\mathcal{G}^{\frS}\), is the graph with vertex set \(\mathfrak G\), where two vertices are adjacent if and only if the corresponding domains are orthogonal.

For each \(U\in\mathfrak G\), let \(\calV_U\) be a set consisting of a single element if \(\calC U\) is quasi-isometric to \(\R\), and a countably infinite set otherwise. 
The \emph{expanded core graph}, denoted $\EOW$, or simply $\EO$ when \(\frS\) is clear, is defined by
\[\calV(\EOW)=\bigsqcup_{U\in\mathfrak G} \calV_U \qquad\text{and}\qquad \calE(\EOW)= \left\{\{v,w\}\mid v\in \calV_U,\ w\in \calV_{U'},\ \text{and}\ U\bot U' \right\}.\]

If, moreover, \(G\) has finite index in a group \(K\), then \((G,\frS)\) is a virtual HHG structure on \(K\) in the sense of \Cref{Def:VirtualHHGStructure}. In this case, we also refer to \(\mathcal G^{\frS}\) and \(\EOW\) as the core graph and the expanded core graph associated to this virtual HHG structure on \(K\).
\end{definition}

\begin{remark}\label{Rmk:VariantsofCoregraph}
In \cite{OP25}, the \emph{orthogonality graph} $\mathcal{O}^{\frS^\infty}$ is the graph whose vertex set is the collection of all unbounded domains, and where two vertices are joined by an edge if and only if the corresponding domains are orthogonal. In particular, our core graph \(\mathcal{G}^{\frS}\) is the subgraph of $\mathcal{O}^{\frS^\infty}$ induced by \(\mathfrak G\).

In \cite{HMS26}, the \emph{minimal orthogonality graph} is defined using the \(\sqsubseteq\)-minimal domains of \(\frS\) as vertices.
By contrast, our core graph uses the minimal elements of the subposet \(\frS^\infty\) of unbounded domains. Thus a minimal unbounded domain need not be \(\sqsubseteq\)-minimal in \(\frS\), since it may contain bounded domains nested below it.
\end{remark}

For structures in \(\Omega\), the vertices of \(\EO\) can be identified with axial directions on minimal unbounded domains, in the following sense.

\begin{definition}\label{Def:AxialDirections}
For an HHG \((G,\frS)\), let \(U\in\frS\) be an unbounded domain. Define
\[\mathscr D(U)=\left\{f\in G\ \middle|\ f \text{ is axial and fully supported on } U \right\}\big/\sim,\]
where $f\sim g$ if and only if $f$ and $g$ have nonzero powers in common.
We refer to an element of \(\mathscr D(U)\) as an \emph{axial direction on \(U\)}.
\end{definition}

\begin{lemma}\label{Lem:AxialDirections}
Suppose that \((G,\frS)\in\Omega\).
Let \(U\in\frG\). Then the following hold.
\begin{enumerate}
\item If \(\calC U\) is a quasi-line, then \(\mathscr D(U)\) consists of a single element.
\item If \(\calC U\) is not a quasi-line, then \(\mathscr D(U)\) is countably infinite.
\end{enumerate}
Consequently, for every \(U\in\frG\), after fixing the set \(\calV_U\) in \Cref{Def:ExpandedCoreGraph}, we may fix a bijection
\[\iota_U:\mathscr D(U)\longrightarrow \calV_U.\]
\end{lemma}

\begin{proof}
The set \(\mathscr D(U)\) is nonempty by \Cref{Def:Omega}\eqref{Omega:AxialDirections}.
If \(\calC U\) is a quasi-line, then any two axial elements fully supported on \(U\) have the same limit set in \(\partial\calC U\). By \Cref{Def:Omega}\eqref{Omega:SameSupportDichotomy}, they have nonzero powers in common. Hence \(\mathscr D(U)\) consists of a single element.

Suppose that \(\calC U\) is not a quasi-line. By \Cref{Def:Omega}\eqref{Omega:AxialDirections}, there exist infinitely many elements fully supported on \(U\) such that every finite subcollection is geometrically irredundant. No two of these elements have nonzero powers in common, so \(\mathscr D(U)\) is infinite. Since \(G\) is countable, \(\mathscr D(U)\) is countably infinite.
\end{proof}

\begin{remark}\label{Rem:RigidAxialDirections}
If \((G,\frS)\in\Xi\), then every axial direction has a rigidly fully supported representative. Indeed, as in the proof of \Cref{Prop:XiContainedInOmega}, every fully supported element has a positive power that is rigidly fully supported on the same domain. For each unbounded \(U\), inclusion therefore induces a bijection between the \(\sim\)-classes of elements rigidly fully supported on \(U\) and \(\mathscr D(U)\).
\end{remark}

In what follows, whenever a group \(G\) is assumed to be virtually in \(\Omega\), we fix a virtual HHG structure \((H,\frS)\in\Omega\) on \(G\). We write \(\EOW\) for its expanded core graph.

\subsection{Expanded core graphs in the motivating examples}\label{Subsection:ExpandedCoreExamples}

We now describe the expanded core graphs of the mapping class group and RAAG structures discussed above.

\subsubsection{The mapping class group case}

We use the convention on essential curves from \Cref{Subsection:Omega}; in particular, boundary-parallel curves are included.

\begin{proposition}\label{Prop:MCGExpandedCore}
Let \(S\) be a connected finite-type orientable surface, possibly with boundary, with \(\chi(S)<0\). For the standard HHG structure on \(\MCG(S)\), the minimal unbounded domains are precisely the annular domains. Thus \(\EO\) is naturally identified with the disjointness graph of essential curves on \(S\): its vertices are isotopy classes of essential simple closed curves, and distinct vertices are adjacent if and only if they admit disjoint representatives.
\end{proposition}
\begin{proof}
In the standard subsurface structure recalled in \Cref{Subsection:Omega}, orthogonality is given by disjointness; see \cite[Theorem~9.2 and Remark~9.3]{HMS26}. Every non-annular unbounded domain contains a properly nested annular domain, and annular curve graphs are quasi-lines \cite[\S2.4]{MM98}.
Thus the minimal unbounded domains are precisely the annuli, each of which contributes a single vertex to \(\EO\) by \Cref{Def:ExpandedCoreGraph}. The correspondence sending an annular domain to the isotopy class of its core curve therefore identifies \(\EO\) with the stated disjointness graph.
\end{proof}

\begin{remark}\label{Rem:CurveGraphConventions}
The conventions for the graph in \Cref{Prop:MCGExpandedCore} differ from those for the usual curve graph in two respects. First, our graph includes isotopy classes of boundary-parallel curves, which are excluded from the vertex set of the usual curve graph. Second, adjacency always means disjointness, without a low-complexity modification. For a once-punctured torus or a four-punctured sphere, the usual curve graph is the Farey graph, with adjacency given by geometric intersection number \(1\) or \(2\), respectively, whereas \(\EO\) is a countably infinite graph with no edges. When \(S\) has no boundary and its usual curve graph \(\calC(S)\) is defined using disjointness, we have \(\EO\cong\calC(S)\).
\end{remark}

\subsubsection{The RAAG case}

We now explain the relation between the expanded core graph associated to an arbitrary rich family and the extension graph of the ambient RAAG.

Let \(G=\bbA(\Gamma)\), and let \((G,\frS)\) be the rich-family HHG structure associated to a rich family \(\mathcal R\). As recalled in \Cref{Subsection:FactorSystems}, domains are the parallelism classes of convex subcomplexes of the form \(g\widetilde S_\Lambda\), where \(\varnothing\ne\Lambda\in\mathcal R\) and \(g\in\bbA(\Gamma)\), or, equivalently, the parallelism classes of the corresponding left cosets \(g\bbA(\Lambda)\).

\begin{proposition}\label{Prop:RAAGExpandedCore}
Let \(G=\bbA(\Gamma)\), equipped with the rich-family HHG structure associated to a rich family \(\mathcal R\). Then there is an induced embedding \(\EOW\rightarrow\Gamma^e\).
Moreover, if \(\mathcal R\) contains every singleton subgraph of \(\Gamma\), then \(\EOW\cong\Gamma^e\).
\end{proposition}

\begin{proof}
We first describe the minimal unbounded domains. Let \(U\in\frG\), and choose a factor-system representative \(F=g\widetilde S_\Lambda\) for some \(\Lambda\in\mathcal R\). We claim that \(\Lambda\) is discrete.

Suppose instead that \(\Lambda\) contains an edge with endpoints \(v\) and \(w\). Since \(\mathcal R\) is a rich family, \(\Lambda'=\Lambda\cap\operatorname{lk}_\Gamma(v)\) belongs to \(\mathcal R\). Moreover, \(\Lambda'\) is a proper subgraph of \(\Lambda\) and contains \(w\). Hence \(F'=g\widetilde S_{\Lambda'}\) represents a domain \(U'\sqsubsetneq U\).
By \Cref{Lem:RichFamilyStabilizers}, \(g\bbA(\Lambda')g^{-1}=\Stab_G(F')\leq G_{U'}\), so \(G_{U'}\) contains the infinite-order element \(gwg^{-1}\).
By \Cref{Prop:Xiclass}, the rich-family HHG structure satisfies the \(\mathbf F_U\) stabilizers property, so \((G_{U'},\frS_{U'})\) is an HHG. Since \(gwg^{-1}\) has infinite order, its bigset in this induced HHG structure is nonempty. Any domain in this bigset is unbounded and nested into \(U'\), and hence is an unbounded domain properly nested into \(U\). 
This contradicts the minimality of \(U\). Therefore \(\Lambda\) is discrete.

Set \(H_U=\Stab_G(F)=g\bbA(\Lambda)g^{-1}\). By \Cref{Lem:RichFamilyStabilizers}, \(H_U\) has finite index in \(G_U\), so its restricted quasi-action on the nesting factor is also cobounded. Define
\[\mathcal W_U=\{h(gvg^{-1})h^{-1} \mid v\in \calV(\Lambda),\ h\in H_U \} \subseteq \calV(\Gamma^e).\] 
We claim that every element of \(\mathcal W_U\), regarded as an element of \(G\), is rigidly fully supported on \(U\).

Indeed, let \(z\in \mathcal W_U\). Then \(z\in H_U\leq G_U\) and \(z\) has infinite order. By \Cref{Prop:AxialEllipticDichotomy} applied to the induced HHG \((G_U,\frS_U)\), the bigset of \(z\) in this induced structure is nonempty. Every domain in this bigset is unbounded and nested into \(U\), and hence, by the minimality of \(U\), the bigset is precisely \(\{U\}\). In particular, \(U\in\B_G(z)\). Since \(z\in G_U\), it follows from \Cref{Lem:RigidSupportCharacterization} that \(z\) is rigidly fully supported on \(U\).

We next compare the cardinalities of \(\mathcal W_U\) and \(\calV_U\). 
\begin{itemize}[wide]
\item If \(\Lambda\) consists of a single vertex, then \(H_U\) is cyclic and \(\mathcal W_U\) consists of a single vertex. Since \(H_U\cong\mathbb Z\) acts coboundedly on \(\calC U\) and contains a loxodromic element, \(\calC U\) is a quasi-line. Hence \(\calV_U\) is also a singleton.
\item Suppose that \(\Lambda\) has at least two vertices. Since \(\Lambda\) is discrete, \(H_U\) is a nonabelian free group, and \(\mathcal W_U\) is countably infinite. Moreover, \(\mathcal W_U\) contains two distinct conjugates of standard generators. By \Cref{Lem:ExtensionGraphPowerRigidity}\eqref{Item:CommonPowersExtensionVertices}, these elements have no nonzero powers in common. Since both are rigidly fully supported on \(U\), \Cref{Lem:RigidlyFullySupportedElements}\eqref{Item:Quasiline} implies that \(\calC U\) is not a quasi-line. Hence \(\calV_U\) is countably infinite.
\end{itemize}
Thus, for every \(U\in\frG\), we may choose a bijection \(\psi_U\colon \calV_U\rightarrow \mathcal W_U\). Taking their union gives a map \(\psi\colon \calV(\EOW)\rightarrow \calV(\Gamma^e)\).

We first show that \(\psi\) is injective. Suppose that \(\psi(x)=\psi(y)=z\), where \(x\in \calV_U\) and \(y\in \calV_V\). By the preceding argument, \(z\) is rigidly fully supported on both \(U\) and \(V\). Hence \(\B(z)=\{U\}=\{V\}\), so \(U=V\). Since \(\psi_U\) is a bijection, it follows that \(x=y\).

It remains to check that \(\psi\) is an induced embedding. Let \(x\in \calV_U\) and \(y\in \calV_V\) be distinct. 
\begin{itemize}[wide]
\item
If \(U\bot V\), then suitable positive powers of \(\psi(x)\) and \(\psi(y)\) commute by \Cref{Lem:VirtualCommutation}, since these elements belong to \(G_U\) and \(G_V\), respectively. As they are conjugates of standard generators, \Cref{Lem:ExtensionGraphPowerRigidity}\eqref{Item:CommutingPowersExtensionVertices} implies that \(\psi(x)\) and \(\psi(y)\) themselves commute. Hence they are adjacent in \(\Gamma^e\).
\item 
Conversely, suppose that \(\psi(x)\) and \(\psi(y)\) are adjacent in \(\Gamma^e\). Then they commute in \(G\). Since they are rigidly fully supported on \(U\) and \(V\), respectively, \Cref{lem:commuting factors} implies that either \(U\bot V\), or \(\psi(x)\) and \(\psi(y)\) have nonzero powers in common. 
In the latter case, \Cref{Lem:ExtensionGraphPowerRigidity}\eqref{Item:CommonPowersExtensionVertices} gives \(\psi(x)=\psi(y)\), contradicting the injectivity of \(\psi\). Therefore \(U\bot V\).
\end{itemize}
Thus, for distinct \(x\in \calV_U\) and \(y\in \calV_V\),
\[\{x,y\}\in \calE(\EOW) \quad\Longleftrightarrow\quad U\bot V \quad\Longleftrightarrow\quad \{\psi(x),\psi(y)\}\in \calE(\Gamma^e),\]
so \(\psi\) is an induced embedding.

Finally, suppose that \(\mathcal R\) contains every singleton subgraph of \(\Gamma\). Since every minimal unbounded domain has a representative \(g\widetilde S_\Lambda\) with \(\Lambda\) discrete, minimality forces \(\Lambda\) to be a singleton. Conversely, every domain represented by a coset of a cyclic standard subgroup is minimal unbounded. Hence the minimal unbounded domains are in bijection with the conjugates of standard generators, and orthogonality is precisely commutation. Therefore the induced embedding above is surjective, and \(\EOW\cong\Gamma^e\).
\end{proof}

\subsection{Extension graphs of intermediate RAAGs}

We now relate the intermediate RAAGs constructed in \Cref{Prop:IntermediateRAAG} to the expanded core graph. More generally, we show that the extension graph of a RAAG generated by fully supported axial elements, with defining graph determined by orthogonality of their supporting domains, embeds into \(\EO\).

\begin{lemma}[Extension graphs in the expanded core graph]\label{lem:intermediate RAAG}
Let \((G,\frS)\in\Omega\), and let \(g_1,\dots,g_m\in G\) be axial elements fully supported on unbounded domains \(U_1,\dots,U_m\in\frS\), respectively.
Suppose that, for some \(N>0\), the subgroup \(M=\langle g_1^N,\dots,g_m^N\rangle\) is a RAAG \(\bbA(\Gamma)\), where \(\Gamma\) has vertex set \(\{v_1,\dots,v_m\}\), the vertex \(v_i\) corresponds to the generator \(g_i^N\), and \(v_i\) is adjacent to \(v_j\) if and only if \(U_i\bot U_j\).
Then the extension graph \(\Gamma^e\) embeds as an induced subgraph of \(\EO\).
\end{lemma}
\begin{proof}
We first construct a second RAAG subgroup with the same defining graph whose standard generators are fully supported on minimal unbounded domains, and then encode conjugates of those generators by axial directions.

We record that if \(U_i=U_j\) for distinct \(i,j\), then \(g_i\) and \(g_j\) cannot have nonzero powers in common. Indeed, otherwise \(g_i^N\) and \(g_j^N\), which correspond to distinct standard generators of \(M\cong\bbA(\Gamma)\), would also have nonzero powers in common, contradicting \Cref{Lem:ExtensionGraphPowerRigidity}\eqref{Item:CommonPowersExtensionVertices}. Hence \(\{g_i,g_j\}\) is geometrically irredundant by \Cref{Def:Omega}\eqref{Omega:SameSupportDichotomy}.
Moreover, if this common domain belongs to \(\frG\), then \(\calC U_i\) is not a quasi-line. Otherwise, \Cref{Lem:AxialDirections} would imply that \(g_i\) and \(g_j\) have nonzero powers in common, a contradiction.


The required replacement is given by the following claim.

\begin{claim}\label{claim:minimal_replacement}
There exist \(V_1,\dots,V_m\in\frG\), with \(V_i\sqsubseteq U_i\) for each \(i\), such that \(V_i\bot V_j\) if and only if \(U_i\bot U_j\) for all distinct \(i,j\). Moreover, if \(V_i=V_j\) for distinct \(i,j\), then \(\calC V_i\) is not a quasi-line.
\end{claim}
\begin{proof}[Proof of the claim]
\renewcommand{\qedsymbol}{$\blacksquare$}
Fix a basepoint \(x_0\in G\). We construct the domains \(V_i\) as follows.
If \(U_i\in\frG\), set \(V_i=U_i\). If \(U_i\notin\frG\), by descending through unbounded domains and using finite complexity, we may choose a domain \(T_i\in\frG\) with \(T_i\sqsubsetneq U_i\), and then replace \(T_i\) by a sufficiently far translate under a power of \(g_i\).

More precisely, for each \(i\) with \(U_i\notin\frG\), let
\[I_i=\{j\mid U_j\pitchfork U_i \text{ or } U_j\sqsubsetneq U_i\}\qquad\text{and}\qquad A_i=\pi_{U_i}(x_0)\cup\bigcup_{j\in I_i}\rho^{U_j}_{U_i}.\]
Since \(I_i\) is finite and each relative projection has uniformly bounded diameter, \(\diam_{\calC U_i}(A_i)<\infty\). Let \(D=\max\left(\{E\}\cup\{\diam_{\calC U_i}(A_i)\mid U_i\notin\frG\}\right)\), where \(E\) is the hierarchy constant.
For each \(i\) with \(U_i\notin\frG\), we then set \(V_i=g_i^{k_i}T_i\), where the exponents \(k_i\) are chosen to satisfy the following requirements:
\begin{itemize}
\item 
Since \(g_i\) fixes \(U_i\) and acts loxodromically on \(\calC U_i\), equivariance gives \(\rho^{g_i^kT_i}_{U_i}=g_i^k\rho^{T_i}_{U_i}\), and these sets eventually leave every bounded subset of \(\calC U_i\). We may therefore choose \(k_i\) so that \(\dist_{\calC U_i}\bigl(\rho^{V_i}_{U_i},A_i\bigr)>10D\).
\item
For indices with the same non-minimal support \(U\), choose the exponents successively so that \(\dist_{\calC U}\bigl(\rho^{V_i}_{U},\rho^{V_j}_{U}\bigr)>10D\) whenever \(i\neq j\). This is possible since each sequence of translated projection sets eventually leaves every bounded subset of \(\calC U\).
\end{itemize}
Since the action of \(G\) preserves nesting and unboundedness, \(V_i\in\frG\) and \(V_i\sqsubsetneq U_i\) whenever \(U_i\notin\frG\). Thus, in all cases, \(V_i\in\frG\) and \(V_i\sqsubseteq U_i\).

If \(U_i\bot U_j\), then \(V_i\bot V_j\), since orthogonality is inherited by nested subdomains. It remains to prove the converse.

\smallskip

\noindent\emph{When \(U_i=U_j\in\frG\).} Then \(V_i=V_j=U_i\), and hence \(V_i\not\bot V_j\). Moreover, as observed before the claim, \(\calC V_i\) cannot be a quasi-line.

\smallskip

\noindent\emph{When \(U_i=U_j=U\notin\frG\).} By construction, \(\dist_{\calC U}(\rho^{V_i}_{U},\rho^{V_j}_{U})>10D\).
If \(V_i\bot V_j\), then the consistency estimate for orthogonal domains (\Cref{Def:HHS}\eqref{Item:Consistency}) gives
\(\dist_{\calC U}(\rho^{V_i}_{U},\rho^{V_j}_{U})\leq 2E\), a contradiction. Thus \(V_i\not\bot V_j\).

\smallskip

\noindent\emph{When \(U_i\sqsubsetneq U_j\).} Then \(U_j\notin\frG\), so \(V_j\sqsubsetneq U_j\), while \(V_i\sqsubseteq U_i\sqsubsetneq U_j\). By consistency for nested domains (or trivially if \(V_i=U_i\)), \(\dist_{\calC U_j}(\rho^{V_i}_{U_j},\rho^{U_i}_{U_j})\leq E\).
Since \(i\in I_j\), the choice of \(V_j\) gives 
\[\dist_{\calC U_j}\bigl(\rho^{V_j}_{U_j},\rho^{U_i}_{U_j}\bigr)>10D.\]
Using \(\diam_{\calC U_j}(\rho^{V_i}_{U_j})\leq E\) and the set-distance estimate from \Cref{Section:Preliminaries}, we obtain
\[\dist_{\calC U_j}\bigl(\rho^{V_i}_{U_j},\rho^{V_j}_{U_j}\bigr)>10D-2E\geq 8D.\]
If \(V_i\bot V_j\), however, the consistency estimate for orthogonal domains gives \(\dist_{\calC U_j}(\rho^{V_i}_{U_j},\rho^{V_j}_{U_j})\leq 2E\), contradicting \(D\geq E\). Thus \(V_i\not\bot V_j\). The case \(U_j\sqsubsetneq U_i\) is symmetric.

\smallskip

\noindent\emph{When \(U_i\pitchfork U_j\).}
We first claim that \(V_i\not\bot U_j\) and \(V_j\not\bot U_i\).
Indeed, if \(V_i=U_i\), then \(V_i\not\bot U_j\) follows immediately from \(U_i\pitchfork U_j\). Otherwise \(V_i\sqsubsetneq U_i\). If \(V_i\bot U_j\), then the consistency estimate for orthogonal domains would give \(\dist_{\calC U_i}(\rho^{V_i}_{U_i},\rho^{U_j}_{U_i})\leq 2E\).
Since \(j\in I_i\), this contradicts the choice \(\dist_{\calC U_i}(\rho^{V_i}_{U_i},\rho^{U_j}_{U_i})>10D\). Thus \(V_i\not\bot U_j\), and similarly \(V_j\not\bot U_i\).

Suppose, toward a contradiction, that \(V_i\bot V_j\). If \(V_i=U_i\) and \(V_j=U_j\), this is impossible since \(U_i\pitchfork U_j\). Thus, after interchanging \(i\) and \(j\) if necessary, we may assume that \(V_i\sqsubsetneq U_i\). Since \(V_i\bot V_j\), the consistency estimate for orthogonal domains gives \(\dist_{\calC U_i}(\rho^{V_i}_{U_i},\rho^{V_j}_{U_i})\leq 2E\). Moreover, \(V_j\sqsubseteq U_j\), \(U_j\pitchfork U_i\), and \(V_j\not\bot U_i\). Hence consistency (or equality of the projection sets if \(V_j=U_j\)) gives \(\dist_{\calC U_i}(\rho^{V_j}_{U_i},\rho^{U_j}_{U_i})\leq E\).
Since \(\diam_{\calC U_i}(\rho^{V_j}_{U_i})\leq E\), the set-distance estimate gives
\[\dist_{\calC U_i}\bigl(\rho^{V_i}_{U_i},\rho^{U_j}_{U_i}\bigr)\leq 2E+E+E=4E.\]
On the other hand, since \(j\in I_i\), the choice of \(V_i\) gives \(\dist_{\calC U_i}(\rho^{V_i}_{U_i},\rho^{U_j}_{U_i})>10D\), a contradiction.

\smallskip

In conclusion, these cases show that if \(V_i\bot V_j\), then \(U_i\bot U_j\) for all distinct \(i,j\), completing the proof of the first assertion.

The arguments above also show that \(V_i=V_j\) can occur only when \(U_i=U_j\in\frG\). Indeed, equality is excluded in the non-minimal equal-support case and in the nested case by the projection-separation estimates above. In the transverse case, if \(V_i=V_j=V\), then \(V\sqsubsetneq U_i,U_j\), and consistency gives \(\dist_{\calC U_i}(\rho^V_{U_i},\rho^{U_j}_{U_i})\leq E\), contrary to the choice of \(V_i\).
As observed before the claim, \(\calC V_i=\calC U_i\) is then not a quasi-line.
\end{proof}

By the final assertion of \Cref{claim:minimal_replacement}, if \(V_i=V_j\) for distinct indices \(i,j\), then \(\calC V_i\) is not a quasi-line. 
Thus, by \Cref{Lem:AxialDirections}, we may choose, for each \(i\), an axial element \(h_i\) fully supported on \(V_i\) so that, whenever \(V_i=V_j\) and \(i\neq j\), the elements \(h_i\) and \(h_j\) represent distinct elements of \(\mathscr D(V_i)\). Then the collection \(h_1,\dots,h_m\) is geometrically irredundant. Indeed, if \(V_i\neq V_j\), then \(\B(h_i)=\{V_i\}\neq\{V_j\}=\B(h_j)\). If \(V_i=V_j\), then \(h_i\) and \(h_j\) represent distinct axial directions and have no nonzero powers in common; in this case geometric irredundancy follows from \Cref{Def:Omega}\eqref{Omega:SameSupportDichotomy}.

By \Cref{Def:Omega}\eqref{Omega:WeakCommutativity}, we may choose a common positive integer \(k\) such that \([h_i^k,h_j^k]=1\) whenever \(V_i\bot V_j\). By \Cref{Prop:SupportedAxialRAAGs}, for every sufficiently large multiple \(L\) of \(k\), the subgroup \(M'=\langle h_1^L,\dots,h_m^L\rangle\le G\) is a RAAG; fix such an \(L\). Its defining graph has vertices indexed by \(1,\dots,m\), with \(i\) adjacent to \(j\) if and only if \(V_i\bot V_j\). By \Cref{claim:minimal_replacement}, this defining graph is naturally identified with \(\Gamma\). The subgroup \(M'\) need not equal \(M\).

We identify \(M'\) with \(\bbA(\Gamma)\) by sending the standard generator \(v_i\) to \(h_i^L\). In what follows, we regard each \(w\in\bbA(\Gamma)\) as its image in \(M'\).
Then we define $\Phi:\calV(\Gamma^e)\rightarrow \calV(\EO)$ by $\Phi(v_i^w)=\iota_{wV_i}\left(\left[wh_i^Lw^{-1}\right]\right)$, where \(\left[wh_i^Lw^{-1}\right]\in\mathscr D(wV_i)\) denotes the axial direction represented by \(wh_i^Lw^{-1}\). Indeed, \(wh_i^Lw^{-1}\) is fully supported on \(wV_i\).

We first show that \(\Phi\) is well-defined. Suppose that \(v_i^w=v_j^{w'}\) as vertices of \(\Gamma^e\). Then the corresponding conjugates of standard generators are equal in \(M'\), and hence $wh_i^Lw^{-1}=w'h_j^L(w')^{-1}$ in \(G\). Therefore the two expressions determine the same domain and the same element of the corresponding set of axial directions. Thus \(\Phi\) is well-defined.

We claim that \(\Phi\) is injective. Suppose that $\Phi(v_i^w)=\Phi(v_j^{w'})$.
Then \(wV_i=w'V_j\), and the elements $wh_i^Lw^{-1}$ and $w'h_j^L(w')^{-1}$ have nonzero powers in common. By \Cref{Lem:ExtensionGraphPowerRigidity}\eqref{Item:CommonPowersExtensionVertices}, the corresponding conjugates of standard generators in \(M'\cong\bbA(\Gamma)\) represent the same vertex of \(\Gamma^e\). Hence $v_i^w=v_j^{w'}$.
Thus \(\Phi\) is injective.

It remains to show that \(\Phi\) is an induced embedding. Let $x=v_i^w$ and $y=v_j^{w'}$ be distinct vertices of \(\Gamma^e\), and write $W=wV_i$ and $W'=w'V_j$.
Suppose first that \(x\) and \(y\) are adjacent in \(\Gamma^e\). Then $wh_i^Lw^{-1}$ and $w'h_j^L(w')^{-1}$ commute in \(M'\), and hence in \(G\). By \Cref{Def:Omega}\eqref{Omega:CommutationDetection}, either these two elements have nonzero powers in common or \(W\bot W'\). The first possibility would imply \(x=y\) by \Cref{Lem:ExtensionGraphPowerRigidity}\eqref{Item:CommonPowersExtensionVertices}, contrary to assumption. Hence \(W\bot W'\), and therefore \(\Phi(x)\) and \(\Phi(y)\) are adjacent in \(\EO\).
Conversely, suppose that \(\Phi(x)\) and \(\Phi(y)\) are adjacent in \(\EO\). Then \(W\bot W'\). By \Cref{Def:Omega}\eqref{Omega:WeakCommutativity}, suitable common positive powers of \(wh_i^Lw^{-1}\) and \(w'h_j^L(w')^{-1}\) commute in \(G\), and hence in \(M'\). Since these two elements correspond to conjugates of standard generators of \(M'\), \Cref{Lem:ExtensionGraphPowerRigidity}\eqref{Item:CommutingPowersExtensionVertices} implies that the elements themselves commute in \(M'\). Therefore the corresponding vertices \(x\) and \(y\) are adjacent in \(\Gamma^e\), which completes the proof of the lemma.
\end{proof}

\section{RAAG embedding obstructions}\label{Section:Obstructions}

We combine the intermediate RAAG factorization from \Cref{Section:IntermediateRAAG} and the extension-graph embedding from \Cref{lem:intermediate RAAG} with the Kim--Koberda obstruction theorem to obtain combinatorial obstructions to RAAG embeddings. We first prove the general clique-graph obstruction and record cases in which the clique graph can be avoided, including the rank-two criterion. We then derive chromatic obstructions and apply them to mapping class groups and virtually compact special groups.

\subsection{The obstruction theorem and sharp cases}

\begin{theorem}[RAAG obstruction]\label{Thm:RAAGObstruction}
Let \(G\) be virtually in \(\Omega\), and let \(\Lambda\) be a finite graph. If \(\bbA(\Lambda)\le G\), then \(\Lambda\le (\EOW)_k\).
\end{theorem}
\begin{proof}
Let \((H,\frS)\in\Omega\) be a virtual HHG structure on \(G\).
Suppose that $\phi:\bbA(\Lambda)\to G$ is an injective homomorphism. By \Cref{Prop:IntermediateRAAG}, after replacing the images of the standard generators of \(\bbA(\Lambda)\) by suitable positive powers, we obtain an embedding of \(\bbA(\Lambda)\) into a RAAG subgroup \(M\le H\) generated by axial elements fully supported on unbounded domains. 
Let \(\Gamma\) be the defining graph of \(M\). By the Kim--Koberda obstruction (\Cref{Thm:KK RAAG emb Thm}), we have $\Lambda\le (\Gamma^e)_k$.
By \Cref{lem:intermediate RAAG}, the extension graph \(\Gamma^e\) embeds as an induced subgraph of \(\EO\). Hence its clique graph \((\Gamma^e)_k\) embeds as an induced subgraph of \(\EO_k\). Therefore \(\Lambda\le \EO_k\).
\end{proof}

\begin{remark}\label{Rem:ClassicalObstructions}
For the standard mapping class group structures, \Cref{Prop:MCGinOmega,Prop:MCGExpandedCore} identify the obstruction in \Cref{Thm:RAAGObstruction} with the clique-graph obstruction in the disjointness graph of essential curves. When \(S\) has no boundary and its usual curve graph is defined using disjointness, this gives
\[\bbA(\Lambda)\le\MCG(S)\quad\Longrightarrow\quad\Lambda\le\calC(S)_k,\]
recovering \cite[Lemma~3.3]{KKObstruction} in this setting; see \Cref{Rem:CurveGraphConventions} for the other surface conventions.
For a RAAG \(\bbA(\Gamma)\) equipped with a rich-family HHG structure, \Cref{Prop:RAAGExpandedCore} gives \(\EO\le\Gamma^e\), so the obstruction specializes to \(\Lambda\le(\Gamma^e)_k\), as in \cite{KK13}.
\end{remark}

We now record cases in which the clique graph can be replaced by \(\EO\) itself.

\begin{corollary}\label{Cor:PositiveCases}
Let \(G\) be virtually in \(\Omega\). Let \(\Lambda\) be a finite graph such that
\(\bbA(\Lambda)\le \bbA(\Gamma)\) implies \(\Lambda\le \Gamma^e\) for every finite graph \(\Gamma\).
If \(\bbA(\Lambda)\le G\), then \(\Lambda\le \EO\).

In particular, the conclusion holds when \(\Lambda\) is a forest or the complement of a linear forest.
\end{corollary}

\begin{proof}
Let \((H,\frS)\in\Omega\) be a virtual HHG structure on \(G\).
By \Cref{Prop:IntermediateRAAG}, we obtain an embedding \(\bbA(\Lambda)\le M\cong\bbA(\Gamma)\le H\) into an intermediate RAAG \(M\). By the assumption on \(\Lambda\), we have \(\Lambda\le \Gamma^e\). By \Cref{lem:intermediate RAAG}, the extension graph \(\Gamma^e\) embeds as an induced subgraph of \(\EO\). Hence \(\Lambda\le \EO\).

The final statement follows from \Cref{Thm:EquivalenceforEmbedding}.
\end{proof}

In rank at most two, the obstruction gives a complete embedding criterion. Recall that the \emph{rank} of an HHS \((\mathcal X,\frS)\) is the maximal cardinality of a collection of pairwise orthogonal unbounded domains, or equivalently, the maximal clique size in the orthogonality graph $\mathcal{O}^{\frS^\infty}$ mentioned in \Cref{Rmk:VariantsofCoregraph}.

\begin{theorem}\label{Thm:RankTwoCriterionHHG}
Let \(G\) be virtually in \(\Omega\) via a virtual HHG structure of rank at most \(2\). Then, for every finite graph \(\Lambda\),
\[\mathbb A(\Lambda)\le G \qquad\Longleftrightarrow\qquad \Lambda\le \EO. \]
\end{theorem}
\begin{proof}
Fix the virtual HHG structure \((H,\frS)\in\Omega\) of rank at most two from the hypothesis.
Suppose first that \(\Lambda\le \EO\). For each vertex \(v\in \calV(\Lambda)\), let \(x_v\in \calV(\EO)\) denote the corresponding vertex of the induced copy of \(\Lambda\). There is a unique minimal unbounded domain \(U_v\in\frG\) such that \(x_v\in \calV_{U_v}\).
Using the bijection \(\iota_{U_v}\colon\mathscr D(U_v)\rightarrow \calV_{U_v}\) from \Cref{Lem:AxialDirections}, choose an axial element \(g_v\) fully supported on \(U_v\) whose axial direction corresponds to \(x_v\).

The collection \(\{g_v\mid v\in \calV(\Lambda)\}\) is geometrically irredundant. Indeed, if \(U_v\neq U_w\), then the two elements have different bigsets. If \(U_v=U_w\), then \(x_v\neq x_w\) implies that \(g_v\) and \(g_w\) represent distinct axial directions and have no nonzero powers in common; in this case geometric irredundancy follows from \Cref{Def:Omega}\eqref{Omega:SameSupportDichotomy}.

Moreover, by the definition of \(\EO\), for distinct vertices \(v,w\), \(\{v,w\}\in \calE(\Lambda)\) if and only if \(U_v\bot U_w\). By \Cref{Def:Omega}\eqref{Omega:WeakCommutativity}, a common positive power of the elements \(g_v\) makes every orthogonal pair commute. Therefore, by \Cref{Prop:SupportedAxialRAAGs}, applied to \((H,\frS)\), suitable common positive powers of the elements \(g_v\) generate a subgroup of \(H\), and hence of \(G\), isomorphic to \(\bbA(\Lambda)\). Thus \(\bbA(\Lambda)\le G\).

Conversely, suppose that \(\bbA(\Lambda)\le G\). By \Cref{Prop:IntermediateRAAG}, after replacing the images of the standard generators of \(\bbA(\Lambda)\) by suitable positive powers, we obtain an embedding \(\bbA(\Lambda)\le M\), where \(M\cong\bbA(\Gamma)\le H\) is an intermediate RAAG generated by axial elements fully supported on unbounded domains.

Since \((H,\frS)\) has rank at most \(2\), the defining graph \(\Gamma\) is triangle-free: a triangle in \(\Gamma\) would correspond to three pairwise orthogonal unbounded supporting domains. Hence \Cref{Thm:EquivalenceforEmbedding} gives \(\Lambda\le \Gamma^e\). 
By \Cref{lem:intermediate RAAG}, the extension graph \(\Gamma^e\) embeds as an induced subgraph of \(\EO\). Therefore \(\Lambda\le \EO\).
\end{proof}

\subsection{Chromatic obstructions}

We now derive a chromatic consequence of the embedding obstructions above. For a graph \(\Lambda\), let \(\chi(\Lambda)\) denote its \emph{chromatic number}, namely the minimum number of colours required for a proper vertex-colouring of \(\Lambda\).

Recall that \(\EOW\) is obtained from the core graph \(\mathcal G^{\frS}\) by replacing each vertex \(U\in\frG\) by the independent set \(\calV_U\), while preserving all edges between the corresponding vertex sets. Hence \(\chi(\EOW)=\chi(\mathcal G^{\frS})\).
In particular, the following consequence of our obstruction results depends only on the chromatic number of the expanded core graph.

\begin{proposition}\label{Prop:ChromaticObstruction}
Let \(G\) be virtually in \(\Omega\), and suppose that \(\chi(\EOW)<\infty\).
Then, for every positive integer \(M\), there exists a finite graph \(\Lambda_M\) of girth at least \(M\) such that \(\bbA(\Lambda_M)\not\le G\).
\end{proposition}

\begin{proof}
By \cite[Lemma~3.2]{KKObstruction}, the clique graph \((\EOW)_k\) also has finite chromatic number. Set \(c=\chi((\EOW)_k)\).
By a theorem of Erd\H{o}s \cite{Erdos1959GraphTheoryProbability} (see also \cite[\S5]{DiestelGraphTheory}), for every positive integer \(M\), there exists a finite graph \(\Lambda_M\) of girth at least \(M\) and chromatic number at least \(c+1\).

Suppose that \(\bbA(\Lambda_M)\le G\). By \Cref{Thm:RAAGObstruction}, we have \(\Lambda_M\le (\EOW)_k\).
Thus \(\chi(\Lambda_M)\leq\chi((\EOW)_k)=c\), contradicting the choice of \(\Lambda_M\). Hence \(\bbA(\Lambda_M)\not\le G\).
\end{proof}

One natural way to ensure the hypothesis \(\chi(\EOW)<\infty\) is through colourability. Recall that an HHG \((G,\frS)\) is called \emph{colourable} if \(\frS\) admits a finite partition into subsets consisting of pairwise transverse domains, and the action of \(G\) on \(\frS\) permutes these subsets, called \emph{colours}. Restricting such a colouring to the minimal unbounded domains gives a proper finite colouring of the core graph, and hence \(\chi(\EOW)<\infty\).
Accordingly, we say that a group \(G\) admits a \emph{colourable virtual HHG structure in \(\Omega\)} if \(G\) has a finite-index subgroup \(H\) equipped with a colourable HHG structure \((H,\frS)\in\Omega\).
We use the analogous terminology for \(\Xi\).

\begin{corollary}\label{Cor:ColourableExamples}
Suppose that \(G\) admits a colourable virtual HHG structure belonging to \(\Omega\). Then, for every positive integer \(M\), there exists a finite graph \(\Lambda_M\) of girth at least \(M\) such that \(\bbA(\Lambda_M)\not\le G\). 
In particular, this conclusion holds for mapping class groups and virtually compact special groups.
\end{corollary}
\begin{proof}
Let \((H,\frS)\) be a colourable virtual HHG structure on \(G\) belonging to \(\Omega\). As observed above, colourability implies \(\chi(\EOW)<\infty\).
The first conclusion therefore follows immediately from \Cref{Prop:ChromaticObstruction}.

For a mapping class group, \Cref{Prop:MCGinOmega} shows that the standard HHG structure belongs to \(\Omega\). After passing to a finite-index subgroup, the restricted standard HHG structure admits a finite equivariant colouring; see \cite{BBF15}. By \Cref{Lem:FiniteIndexXiOmega}, the restricted HHG structure still belongs to \(\Omega\). Thus the mapping class group admits a colourable virtual HHG structure in \(\Omega\).

Now suppose that \(G\) is virtually compact special. Choose a finite-index subgroup \(H\le G\) which is the fundamental group of a compact special cube complex, and equip \(H\) with a rich-family HHG structure as in \Cref{Subsection:FactorSystems}. By \Cref{Prop:Xiclass}, this structure belongs to \(\Xi\).
Moreover, these cubical HHG structures virtually admit BBF colourings; see \cite[Lemma~11.3]{BHS17I} and \cite{HP22}. Thus, after passing to a further finite-index subgroup \(H_1\le H\), the restricted HHG structure is colourable. By \Cref{Lem:FiniteIndexXiOmega}, this restricted structure still belongs to \(\Xi\). By \Cref{Prop:XiContainedInOmega}, the restricted structure belongs to \(\Omega\), so the conclusion follows from the first part of the corollary.
\end{proof}

For virtually compact special groups, this conclusion also follows from their virtual embeddings into RAAGs and the Kim--Koberda obstruction; the argument here derives it directly from the chosen hierarchical structure. The mapping class group case recovers \cite[Theorem~1.2]{KKObstruction}.

\section{Further questions}\label{Section:FurtherQuestions}

We conclude with questions about the scope of \(\Xi\) and \(\Omega\), the strength of the obstructions obtained from the expanded core graph, and their dependence on the chosen HHG structure.

We begin with the scope of the two classes. In our previous work \cite{OP25}, we isolated several families of HHG structures in terms of the existence of suitably supported axial elements and the commutation of suitable powers of elements supported on orthogonal domains. Several standard short HHGs, and more generally combinatorial HHGs, provide natural candidates for the present framework. The converse problem studied here also requires conditions on decomposition and commutation that allow us to associate an intermediate RAAG to an arbitrary RAAG embedding.

\begin{question}\label{Qn:ShortCombinatorialHHGs}
Which short HHGs, or more generally which combinatorial HHGs, admit HHG structures belonging to \(\Omega\), and which admit structures in the subclass \(\Xi\)? In particular, which of the structures considered in \cite{OP25} satisfy the decomposition and commutation conditions required in the present paper?
\end{question}

\Cref{Prop:ChromaticObstruction} requires only that the expanded core graph have finite chromatic number, while colourability of the HHG structure provides one natural sufficient condition for this finiteness. 
These two conditions, however, concern different parts of the HHG structure: the chromatic number of the expanded core graph depends only on orthogonality among minimal unbounded domains, whereas colourability is a condition on the entire index set together with the group action.  
Moreover, non-colourable HHGs are known to exist \cite{Hag23}. This raises the following questions.

\begin{question}\label{Qn:ColourabilityAndChromaticNumber}
Let \((G,\frS)\) belong to \(\Omega\).
Under what additional hypotheses does finite chromatic number of \(\EOW\) imply colourability of the HHG structure? Does either class \(\Xi\) or \(\Omega\) contain a non-colourable HHG structure, and can such a structure still satisfy \(\chi(\EOW)<\infty\)?
\end{question}

We ask whether the complete embedding criterion of \Cref{Thm:RankTwoCriterionHHG} extends beyond rank two. In higher rank, \Cref{Thm:RAAGObstruction} only gives an obstruction in the clique graph.
For a RAAG \(\bbA(\Gamma)\) equipped with a rich-family structure satisfying \(\EO\cong\Gamma^e\), the desired implication is the converse to the Extension Graph Theorem. The converse fails in general \cite{CDK13,LL18}, although several positive cases are known; see \Cref{Thm:EquivalenceforEmbedding}.

\begin{question}\label{Qn:RemovingCliqueGraph}
For which HHG structures \((G,\frS)\) belonging to \(\Omega\) does
\[\bbA(\Lambda)\le G\qquad\Longleftrightarrow\qquad\Lambda\le \EOW\]
hold for every finite graph \(\Lambda\)?
More generally, for which finite graphs \(\Lambda\), or under what conditions on \((G,\frS)\), can the obstruction \(\Lambda\le (\EOW)_k\) be improved to \(\Lambda\le \EOW\)?
\end{question}

We finally consider the dependence of the expanded core graph on the chosen HHG structure. Changing the structure may change its minimal unbounded domains and their orthogonality relations. The RAAG case provides a natural test: by \Cref{Prop:RAAGExpandedCore}, the expanded core graph of every rich-family HHG structure on \(\bbA(\Gamma)\) embeds as an induced subgraph of \(\Gamma^e\).

\begin{question}[Dependence on the HHG structure]\label{Qn:DependenceOnStructure}\label{Qn:UniversalExpandedCoreGraph}\label{Qn:RAAGUniversalExpandedCoreGraph}
Let \(G\) admit a virtual HHG structure in \(\Omega\). Which features of the expanded core graph are invariant under changing this structure, including the choice of finite-index subgroup?

Does some virtual HHG structure on \(G\) in \(\Omega\) have an expanded core graph containing those of all other such structures as induced subgraphs? More generally, does this collection have a maximal element with respect to induced embeddability, and can one be chosen canonically from \(G\)?
In particular, if \(G=\bbA(\Gamma)\), does \(\EOW\le\Gamma^e\) hold for every HHG structure \((G,\frS)\in\Omega\)?
\end{question}

\appendix

\section{Permanence properties of the classes \texorpdfstring{\(\Xi\)}{Xi} and \texorpdfstring{\(\Omega\)}{Omega}}\label{Appendix:OmegaPermanence}

We prove the permanence statements from the main text. For each construction, we first record the common geometric facts and then verify the conditions specific to \(\Xi\) and \(\Omega\).

\subsection{Finite-index subgroups}\label{Appendix:FiniteIndex}

We first recall the elementary inheritance statement used in the proof.

\begin{lemma}\label{Lem:FiniteIndexHHG}
Let \((G,\frS)\) be an HHG. Every finite-index subgroup \(H\le G\) admits the inherited HHG structure, which we denote by \((H,\frS)\).
\end{lemma}
\begin{proof}
Equip \(H\) with a word metric.
Since \(H\) has finite index in \(G\), the inclusion \(H\hookrightarrow G\) is a quasi-isometry, so restricting the projections to \(H\) gives an HHS structure \((H,\frS)\).
The restricted domain and coordinate actions retain the required equivariance, and each \(G\)-orbit in \(\frS\) splits into finitely many \(H\)-orbits.
Thus \((H,\frS)\) is an HHG.
\end{proof}

\begin{proof}[Proof of \Cref{Lem:FiniteIndexXiOmega}]
The inherited structure retains the hierarchical relations, coordinate spaces, and relative projections, while the projections from \(G\) and the prescribed actions are restricted to \(H\).
Thus, for every \(U\in\frS\), we have \(H_U=H\cap G_U\).
In particular, \(H_U\) has finite index in \(G_U\). Full support, rigid full support, and geometric irredundancy are unchanged for elements of \(H\).

We first verify the decomposition condition for both classes. Let \(h\in H\) have infinite order, and write \(\B(h)=\{U_1,\dots,U_k\}\). In either class, there exist \(n>0\) and pairwise commuting axial elements \(g_1,\dots,g_k\in G\) such that \(h^n=g_1\cdots g_k\), where \(g_i\) is fully supported on \(U_i\). In the \(\Xi\) case, choose the pairwise commuting factors supplied by orthogonal decomposition, so that each factor is rigidly fully supported. Choose a common positive integer \(m\) such that every \(g_i^m\) belongs to \(H\). Then
\[h^{nm}=g_1^m\cdots g_k^m.\]
Positive powers preserve full support and rigid full support, so this proves the required decomposition condition in both cases.

Suppose that \((G,\frS)\in\Xi\). Since \(H_U\) has finite index in \(G_U\), its induced quasi-action on \(\mathbf F_U\) remains proper and cobounded, and \Cref{Lem:FiniteIndexHHG} gives the inherited HHG structure \((H_U,\frS_U)\) with the restricted domain and coordinate actions. Thus the \(\mathbf F_U\) stabilizers property passes to \(H\). Together with the decomposition above, this proves \((H,\frS)\in\Xi\).

Now suppose that \((G,\frS)\in\Omega\). Every axial element fully supported on an unbounded domain has a positive power in \(H\), and every such power is fully supported on the same domain. Applying this observation to an infinite geometrically irredundant collection preserves geometric irredundancy, so \Cref{Def:Omega}\eqref{Omega:AxialDirections} holds for \(H\). The same-support dichotomy concerns the same elements and hierarchical data, and the powers supplied by weak commutativity already belong to \(H\). Commutation detection then follows from the observation after \Cref{Def:Omega}. Together with the decomposition above, this proves \((H,\frS)\in\Omega\).
\end{proof}

\subsection{Finite direct products}\label{Appendix:OmegaDirectProducts}

Direct products of HHGs are HHGs by \cite[Proposition~8.27 and Corollary~8.28]{BHS19II}. We use a variant of that construction, retaining the factor domains and their coordinate spaces but choosing a different collection of auxiliary domains with one-point coordinate spaces. These are the mixed domains in the rectangular index set below. Their nesting factors are products of the original nesting factors, with induced HHS structures given by the same rectangular construction. This allows us to verify the \(\mathbf F_U\) stabilizers property for every domain, including the bounded ones.

\begin{lemma}[Rectangular product structure]\label{Lem:RectangularProductStructure}
Let \((G_i,\frS_i)\), for \(i=1,2\), be HHGs. For each \(i\), adjoin a formal minimum \(0\) to the nesting order, so that \(0\sqsubseteq U\) for every \(U\in\frS_i^0:=\frS_i\cup\{0\}\), and set
\[\frS=\frS_1\boxtimes\frS_2=\bigl(\frS_1^0\times\frS_2^0\bigr)\setminus\{(0,0)\}.\]
Then \(G_1\times G_2\) admits an HHG structure with index set \(\frS=\frS_1\boxtimes\frS_2\) and coordinatewise nesting. For \(A=(A_1,A_2)\) and \(B=(B_1,B_2)\) in \(\frS\), we have \(A\bot B\) if and only if \(A_i\bot B_i\) for every \(i\) such that both \(A_i\) and \(B_i\) are nonzero.

Call an element of \(\frS\) \emph{pure} if exactly one of its components is nonzero, and \emph{mixed} if both components are nonzero.
The coordinate space of a pure domain is inherited from its nonzero component, while every mixed domain has a one-point coordinate space. The induced HHS structure on the nesting factor of a pure domain is the corresponding induced structure in the original factor. At a mixed domain, it is obtained by applying the same rectangular construction to the two corresponding induced structures.
\end{lemma}
\begin{proof}
Write \(G=G_1\times G_2\), and equip it with the \(\ell_1\) product metric
\[\dist_G\bigl((g_1,g_2),(h_1,h_2)\bigr)=\dist_{G_1}(g_1,h_1)+\dist_{G_2}(g_2,h_2).\]
Let \(\operatorname{pr}_i\colon G\to G_i\) be the factor projection, and write \(\pi_U^i\) for the coordinate projection associated to \(U\in\frS_i\). For \(U\in\frS_1\) and \(V\in\frS_2\), define
\[\pi_{(U,0)}=\pi_U^1\circ\operatorname{pr}_1\qquad\text{and}\qquad\pi_{(0,V)}=\pi_V^2\circ\operatorname{pr}_2.\]
Projections to mixed domains are the unique maps to their one-point coordinate spaces.

We use the nesting and orthogonality relations specified in the statement, and declare distinct domains transverse when neither is nested into the other and they are not orthogonal. The maximal domain is \((S_1,S_2)\), where \(S_i\) is the maximal domain of \(\frS_i\). In particular, \((U,0)\bot(0,V)\) for every \(U\in\frS_1\) and \(V\in\frS_2\). Relative projections between distinct non-orthogonal pure domains from the same factor are inherited from that factor, and every relative projection whose target is the coordinate space of a mixed domain is the unique map to its point. If a mixed domain \(D=(D_1,D_2)\in\frS\) is transverse to a pure domain \((W,0)\), with \(W\in\frS_1\), set \(\rho_{(W,0)}^D=\rho_W^{D_1}\), where the right-hand side is the relative projection in the first factor. 
Indeed, the prescribed relations imply that either \(D_1\sqsubsetneq W\) or \(D_1\pitchfork W\), so this bounded relative projection is defined. Use the symmetric definition for a pure domain from the second factor. For each pure domain \(A\in\frS\), choose \(p_A\in\pi_A(1_G)\), where \(1_G=(1_{G_1},1_{G_2})\) is the identity of \(G\). 
Whenever a pure domain \(A\) is properly nested into a mixed domain \(D\), define the downward relative projection \(\rho_A^D\colon\calC D\to2^{\calC A}\) to be the constant map with value \(\{p_A\}\). These choices do not affect consistency or bounded geodesic image, since the upper coordinate space is a point.

We now verify the HHS axioms. Hyperbolicity and the projection axioms are inherited or immediate for one-point spaces. Coordinatewise nesting has finite complexity, bounded in terms of the two original complexities, and the prescribed orthogonality relation is hereditary under nesting. For the container axiom, let \(A=(A_1,A_2),T=(T_1,T_2)\in\frS\) satisfy \(A\sqsubseteq T\), and suppose that
\[\{B\in\frS:B\sqsubseteq T,\ B\bot A\}\ne\emptyset.\]
For each \(i\), choose \(C_i\in\frS_i^0\) as follows. If \(A_i=0\), take \(C_i=T_i\). If \(A_i\ne0\) and there is a domain \(B_i\in\frS_i\) with \(B_i\sqsubseteq T_i\) and \(B_i\bot A_i\), take a container \(C_i\sqsubsetneq T_i\) supplied by the factor axiom. Otherwise, take \(C_i=0\). The assumed nonempty orthogonal complement ensures that \(C=(C_1,C_2)\ne(0,0)\). Since \(A\ne(0,0)\), at least one component satisfies \(A_i\ne0\), and then \(C_i\sqsubsetneq T_i\). Thus \(C\in\frS\), \(C\sqsubsetneq T\), and every \(B\sqsubseteq T\) orthogonal to \(A\) is nested into \(C\).

The pointwise consistency inequalities and bounded geodesic image reduce to the factor axioms whenever the relevant upper or transverse domains are pure. Every other relevant inequality has a zero term in the coordinate space of a mixed domain. For the consistency of bounded relative projections, an inequality whose target is the coordinate space of a pure domain is the corresponding inequality in that factor, or is an equality when the two source components agree; an inequality whose target is the coordinate space of a mixed domain is immediate. The large links axiom for a pure domain follows from the corresponding factor axiom. For a mixed domain \((U,V)\in\frS\), the two proper subdomains \((U,0)\) and \((0,V)\) contain every nested domain with a possibly large projection distance. Thus they supply large links after increasing the uniform constant.

For partial realization, take a finite pairwise orthogonal collection of domains in \(\frS\) and consider its nonzero components in each factor. By the prescribed orthogonality relation, these components are pairwise orthogonal in the corresponding factor. Prescribe the required points for the components of pure domains and choose arbitrary points in the projection images for the components of mixed domains. Partial realization in each factor yields a point of \(G\) satisfying the required projection estimates. The estimates in pure domains containing, or transverse to, a prescribed rectangular domain are exactly the corresponding factor estimates; the estimates in mixed domains are automatic. Finally, uniqueness follows from uniqueness in the two factors, using the displayed formulas for the pure-domain projections. This proves the HHS axioms.

The product action preserves the hierarchical relations and is equivariant for projections and bounded relative projection sets. It is proper and cobounded, and it has finitely many orbits on \(\frS\). Hence this is an HHG structure.

The domains nested into a pure domain, together with their coordinate spaces and projections, recover the induced HHS structure on the corresponding nesting factor of the original structure. For a mixed domain \((U,V)\), the nested index set is \((\frS_1)_U\boxtimes(\frS_2)_V\), and the restricted data are given by the same rectangular construction. This proves the assertion about the induced HHS structures.
\end{proof}

\begin{remark}[Standard factors of the rectangular product]\label{Rem:RectangularProductFactors}
Let \(U\in\frS_1\) and \(V\in\frS_2\). We use superscripts \(i\) to distinguish the standard factors in the original structures \((G_i,\frS_i)\), and write \(F_U^1,E_U^1\) and \(F_V^2,E_V^2\) for chosen geometric models of \(\mathbf F_U^1,\mathbf E_U^1\) and \(\mathbf F_V^2,\mathbf E_V^2\), respectively. Using the groups \(G_i\), equipped with word metrics, as the ambient models, we may choose geometric models \(F_A,E_A\) of the rectangular standard factors by setting
\[
\begin{aligned}
F_{(U,0)}&=F_U^1,& E_{(U,0)}&=E_U^1\times G_2,\\
F_{(0,V)}&=F_V^2,& E_{(0,V)}&=G_1\times E_V^2,\\
F_{(U,V)}&=F_U^1\times F_V^2,& E_{(U,V)}&=E_U^1\times E_V^2.
\end{aligned}
\]
These equalities specify the geometric models; the spaces of partial tuples are uniformly coarsely identified with these models. Indeed, the pure coordinates separate into the two original factors, mixed domains have one-point coordinate spaces, and consistent tuples in each original structure are uniformly coarsely realized in \(G_i\).
\end{remark}

For both permanence arguments, write \(G=G_1\times G_2\) and \(\frS=\frS_1\boxtimes\frS_2\), and identify the pure domains with their corresponding factor domains. All unbounded domains are pure, so, for every \(g=(g_1,g_2)\in G\),
\begin{equation}\label{Eq:ProductBigsetOmega}
\B_G(g)=\B_{G_1}(g_1)\sqcup\B_{G_2}(g_2).
\end{equation}
The canonical inclusions \(g_1\mapsto(g_1,1_{G_2})\) and \(g_2\mapsto(1_{G_1},g_2)\) preserve full support, rigid full support, and geometric irredundancy. Suppose that \(f=(f_1,f_2)\in G\) is fully supported on the pure domain \((U,0)\), where \(U\in\frS_1\). Then \(f_1\) is fully supported on \(U\) with respect to \((G_1,\frS_1)\), and \(f_2\) has finite order. Indeed, the action on the first factor is inherited, and \eqref{Eq:ProductBigsetOmega} gives \(\B_{G_2}(f_2)=\emptyset\), so \Cref{Prop:AxialEllipticDichotomy} applies. The analogous statement holds for \(\frS_2\). Thus a suitable positive power of every fully supported element lies in the corresponding direct factor.

We also use the same decomposition construction in both cases. Suppose that suitable positive powers of the infinite-order coordinates of \(g=(g_1,g_2)\) admit decompositions into pairwise commuting factors. Choose \(n>0\) so that the \(n\)-th powers of the finite-order coordinates are trivial and those of the infinite-order coordinates admit such decompositions. Write \(g_1^n=h_1\cdots h_r\) and \(g_2^n=k_1\cdots k_s\), allowing empty products. Then
\[
g^n=(h_1,1)\cdots(h_r,1)(1,k_1)\cdots(1,k_s),
\]
where the displayed factors pairwise commute. Their supporting domains, when the coordinate decompositions are supported on the respective bigsets, are precisely the domains in \(\B_G(g)\), by \eqref{Eq:ProductBigsetOmega}.

\begin{proof}[Proof of \Cref{Prop:Xiclass}\eqref{Item:XiDirectProducts}]
For \(i=1,2\), define
\[
N_i=\{g\in G_i:gW=W,\ g^\diamond(W)=\id_{\calC W}\text{ for every unbounded }W\in\frS_i\}.
\]
Each \(N_i\) is finite. Indeed, equivariance bounds the coordinate displacement of a fixed basepoint uniformly on every unbounded domain, while the bounded coordinate spaces have uniformly bounded diameter. The uniqueness axiom therefore bounds the orbit of that basepoint under \(N_i\), and properness gives finiteness.

Since all unbounded domains of the rectangular structure are pure, the definition of the orthogonal stabilizers gives
\begin{equation}\label{Eq:RectangularProductStabilizers}
G_{(U,0)}=(G_1)_U\times N_2,\quad
G_{(0,V)}=N_1\times(G_2)_V,\quad
G_{(U,V)}=(G_1)_U\times(G_2)_V.
\end{equation}
For example, the unbounded domains orthogonal to \((U,0)\) are the unbounded first-factor domains orthogonal to \(U\), together with every unbounded domain from the second factor. The other two formulas follow in the same way.

By \Cref{Lem:RectangularProductStructure}, the induced HHS structure on the nesting factor at \((U,0)\) is the original induced structure on \(\mathbf F_U^1\). The subgroup \(N_2\) acts trivially on every domain and coordinate space of this induced structure. Up to uniformly bounded error, the induced quasi-action on \(\mathbf F_U^1\) is therefore the pullback of the original induced quasi-action under the projection \((G_1)_U\times N_2\to(G_1)_U\). This projection has finite kernel, so properness and coboundedness are preserved, and the domain action still has finitely many orbits. The same argument applies to \((0,V)\). At a mixed domain \((U,V)\), the rectangular product of the two induced HHG structures gives the required structure, by \Cref{Lem:RectangularProductStructure}. This proves the \(\mathbf F_U\) stabilizers property for every domain, including the bounded ones.

For an infinite-order element \(g\in G\), apply the common decomposition construction above to the commuting rigidly fully supported factors supplied by orthogonal decomposition in the two factors. Their images under the canonical inclusions remain rigidly fully supported. Together with \eqref{Eq:ProductBigsetOmega}, this proves orthogonal decomposition in \(G\), and hence \((G,\frS)\in\Xi\).
\end{proof}

\begin{proof}[Proof of \Cref{Prop:OmegaPermanence}\eqref{Item:OmegaDirectProducts}]
We verify the three conditions in \Cref{Def:Omega}. Since the canonical inclusions preserve full support and geometric irredundancy, the existence condition \eqref{Omega:AxialDirections} follows from the corresponding conditions in the factors. Applying the common decomposition construction above to their fully supported decompositions proves \eqref{Omega:Decomposition}.

For \eqref{Omega:CommutationPackage}, let \(f,h\in G\) be fully supported on unbounded domains. By the observations above, there exists \(n>0\) such that each of \(f^n\) and \(h^n\) lies in the direct factor corresponding to its supporting domain. If the supports belong to different factors, then they are orthogonal and \([f^n,h^n]=1\), so weak commutativity and commutation detection hold in this case.

Suppose that both supports belong to the same factor. If they coincide, then the same-support dichotomy in that factor applies to \(f^n\) and \(h^n\). 
Since both geometric irredundancy and the property of having nonzero powers in common are unchanged by taking positive powers, this proves \eqref{Omega:SameSupportDichotomy}. If the supports are orthogonal, then weak commutativity in that factor gives some \(m>0\) such that \([f^{nm},h^{nm}]=1\), proving \eqref{Omega:WeakCommutativity}. Commutation detection follows from the same-support dichotomy by the observation after \Cref{Def:Omega}. Hence the product structure belongs to \(\Omega\), and iteration gives finite direct products.
\end{proof}

\subsection{Relative hyperbolicity}\label{Appendix:OmegaRelativelyHyperbolic}

We first record the features of the relatively hyperbolic construction used in both proofs. Let \(G\) be hyperbolic relative to a finite collection \(\mathcal P\), where each \(P\in\mathcal P\) is equipped with an HHG structure \((P,\frS_P)\), and equip \(G\) with the relatively hyperbolic HHG structure of \cite[Theorem~9.1]{BHS19II}. Let \(S\) be its maximal domain.

Since \(\frS_S=\frS\) and no domain is orthogonal to \(S\), we have \(G_S=G\), and the standard realization map \(\phi_S\colon\mathbf F_S\to G\) is a quasi-isometry compatible with the coordinate projections and coarsely equivariant for the induced \(G\)-quasi-action.
For every other domain \(aV\in a\frS_P\), its stabilizer preserves the peripheral-coset label and is therefore contained in \(aPa^{-1}\). All domains orthogonal to \(aV\) belong to the same peripheral copy, where the coordinate data and action are inherited from \((P,\frS_P)\). Consequently, \(G_{aV}=aG_V^Pa^{-1}\), where \(G_V^P\) denotes the orthogonal stabilizer for the peripheral structure. The nesting factors and their induced quasi-actions are inherited from that structure as well.

For a peripheral copy \(a\frS_P\), set \(Q=aPa^{-1}\). If \(U\in a\frS_P\) and \(f\in G\) fixes \(U\), then \(f\in Q\), and conjugation by \(a^{-1}\) identifies the HHG structure on \(Q\) induced by \(a\frS_P\) with \((P,\frS_P)\). The subgroup \(Q\) preserves the coset \(aP\), whose projections to \(\calC S\) and to the coordinate spaces in every other peripheral copy are bounded in this construction. Consequently, for \(f\in Q\), the active domains of \(f\) in \((G,\frS)\) are precisely the translates by \(a\) of the active domains of \(a^{-1}fa\) in \((P,\frS_P)\). Thus full support is inherited from the peripheral structure, and the stabilizer formula above gives the same conclusion for rigid full support.

Let \(g\in G\) have infinite order. If \(g\) acts loxodromically on \(\calC S\), then \(\B(g)=\{S\}\), so \(g\) is rigidly fully supported on \(S\) and gives a one-factor decomposition. Otherwise, \cite[Theorem~1.14]{Osi06} implies that \(g\) is conjugate into a peripheral subgroup. Write \(p=a^{-1}ga\in P\). Any decomposition \(p^n=h_1\cdots h_k\) into elements fully supported on the domains \(U_i\in\B_P(p)\) gives
\[g^n=(ah_1a^{-1})\cdots(ah_ka^{-1})\quad\text{and}\quad
\B_G(g)=\{aU_1,\dots,aU_k\}.\]
The factors are fully supported on the corresponding domains \(aU_i\). Rigid full support and pairwise commutation are also preserved by this conjugation. We use this decomposition argument in both cases below.

\begin{proof}[Proof of \Cref{Prop:Xiclass}\eqref{Item:HyperbolicRelto}]
Suppose that each peripheral structure belongs to \(\Xi\). The description of the nesting factors and stabilizers above shows that their induced HHG structures are inherited from the peripheral structures, while the maximal nesting factor carries the original HHG structure on \(G\). Thus the \(\mathbf F_U\) stabilizers property holds.

The orthogonal decomposition property follows from the preceding decomposition argument, applied to the pairwise commuting rigidly fully supported peripheral factors. Therefore \((G,\frS)\in\Xi\).
\end{proof}

\begin{proof}[Proof of \Cref{Prop:OmegaPermanence}\eqref{Item:OmegaRelativelyHyperbolic}]
Suppose that each peripheral structure belongs to \(\Omega\). We verify the three conditions in \Cref{Def:Omega} using the same structural description.

\smallskip

\noindent\emph{\Cref{Def:Omega}\eqref{Omega:AxialDirections}.}
Let \(U\in\frS\) be unbounded. If \(U\in a\frS_P\), the conclusion follows from the corresponding peripheral condition and the inheritance of full support.

Suppose that \(U=S\). By \Cref{Lem:ExistenceOfAxialElement}, there exists an axial element \(g\in G\) with \(\B(g)=\{S\}\), and hence \(g\) is fully supported on \(S\).

If \(\calC S\) is not a quasi-line, then the cobounded acylindrical action \(G\curvearrowright\calC S\) is non-elementary. Indeed, it contains a loxodromic element, and an elementary cobounded action with a loxodromic element would force \(\calC S\) to be a quasi-line. Hence, by \cite[Theorem~1.1]{Osi16}, the action contains infinitely many pairwise independent loxodromic elements. Since no domain is orthogonal to \(S\), each of these elements has bigset \(\{S\}\) and is therefore fully supported on \(S\). Their pairwise disjoint limit sets imply that every finite subcollection is geometrically irredundant.

\smallskip

\noindent\emph{\Cref{Def:Omega}\eqref{Omega:Decomposition}.}
Apply the decomposition argument above to the pairwise commuting fully supported factors supplied by the peripheral \(\Omega\)-structures. Their conjugates remain pairwise commuting and have precisely the required supporting domains.

\smallskip

\noindent\emph{\Cref{Def:Omega}\eqref{Omega:CommutationPackage}.}
First, we verify \Cref{Def:Omega}\eqref{Omega:SameSupportDichotomy}. Let \(f_1,f_2\) be fully supported on the same unbounded domain \(U\).
If \(U\in a\frS_P\), the conclusion follows from the corresponding peripheral condition. If \(U=S\), then \(f_1\) and \(f_2\) are loxodromic for the acylindrical action on \(\calC S\). If their limit sets are disjoint, the pair is geometrically irredundant. Otherwise, they lie in the same virtually cyclic elementary subgroup \cite[\S6]{Osi16}, and hence have nonzero powers in common.

Next, let \(f_i\) be fully supported on \(U_i\), for \(i=1,2\). If \(U_1\bot U_2\), then the two domains belong to a common peripheral copy \(a\frS_P\). Thus \(f_1,f_2\in aPa^{-1}\), and their conjugates are fully supported on orthogonal domains of \((P,\frS_P)\). Weak commutativity in the peripheral structure, followed by conjugation, gives \([f_1^N,f_2^N]=1\) for some \(N>0\).

Commutation detection follows from the same-support dichotomy by the observation after \Cref{Def:Omega}. Hence \((G,\frS)\in\Omega\).
\end{proof}

\section{Comparison with the Abbott--Behrstock framework}\label{Appendix:ComparisonWithAB}

Following \cite[\S3]{AB23}, we say that an HHS has \emph{clean containers} if, for every finite collection \(\mathcal U=\{U_1,\dots,U_k\}\) of pairwise orthogonal domains that is not maximal with respect to inclusion, there exists a domain \(C_{\mathcal U}\) such that every domain orthogonal to each \(U_i\) is nested into \(C_{\mathcal U}\), and \(C_{\mathcal U}\bot U_i\) for every \(i\). When \(\mathcal U=\{U\}\), we write \(C_U\).
This condition is part of the Abbott--Behrstock framework, but is not assumed in our definition of \(\Xi\). If \(U\) has a clean container \(C_U\), then \(\frS_U^\bot=\frS_{C_U}\), so \(\mathbf E_U\) is identified with \(\mathbf F_{C_U}\) as a space of partial tuples; see \cite[Lemma~3.6]{AB23}.

In \cite[\S1.3]{AB23}, the standard factors are identified with their geometric realizations. For this comparison, work in an ambient HHS \((\mathcal X,\frS)\) with a proper cobounded isometric \(G\)-action, and fix the geometric realizations used in their definition. Denote by \(G_U^{\mathrm{AB}}\) the subgroup of \(\Stab_G(U)\) preserving every resulting geometric \(\mathbf F_U\)-fiber, as in \cite[Definition~3.1]{AB23}. We take these families to be equivariant under the ambient action and to represent every \(\mathbf e\in\mathbf E_U\): for each such partial tuple, there is a chosen fiber at uniformly bounded Hausdorff distance from \(\phi_U(\mathbf F_U\times\{\mathbf e\})\), with a bound independent of \(U\) and \(\mathbf e\). Our subgroup \(G_U\), defined in \Cref{Def:OurGUandFU}, instead fixes the unbounded domains orthogonal to \(U\) and acts trivially on their coordinate spaces; it requires no additional choice of geometric fibers. The definitions do not give an inclusion in either direction between \(G_U^{\mathrm{AB}}\) and \(G_U\).

Preserving every geometric fiber gives uniform control on the orthogonal coordinate action. The following observation does not require either version of the \(\mathbf F_U\) stabilizers property.

\begin{lemma}[Bounded orthogonal action]\label{Lem:ABBoundedOrthogonalAction}
There is a uniform constant \(D\) such that, for every \(U\in\frS\), \(h\in G_U^{\mathrm{AB}}\), and unbounded \(W\bot U\), we have
\[hW=W\quad\text{and}\quad \dist_{\calC W}(p,h^\diamond(W)p)\le D\quad\text{for every }p\in\calC W.\]
In particular, if \(g\in G_U^{\mathrm{AB}}\) and \(U\in\B(g)\), then \(g\) is fully supported on \(U\).
\end{lemma}
\begin{proof}
For \(x\in\mathcal X\), restrict its projection tuple to the domains orthogonal to \(U\), and choose a corresponding geometric \(\mathbf F_U\)-fiber. Its orthogonal projections are uniformly close to the corresponding coordinates of \(x\). Since \(h\) preserves this fiber, coarse equivariance gives a uniform constant \(D_0\) such that
\begin{equation}\label{Eq:ABBoundedOrthogonalCoordinates}
\operatorname{Haus}_{\calC(hW)}\bigl(h^\diamond(W)\pi_W(x),\pi_{hW}(x)\bigr)\le D_0
\end{equation}
for every \(x\in\mathcal X\). The constant is independent of \(h,U,W\).

Suppose that \(hW\ne W\). These domains cannot be nested, since iterating a strict nesting under \(h\) or \(h^{-1}\) would contradict finite complexity. If \(W\bot hW\), partial realization allows their coordinates to be prescribed independently up to a uniform error. Keeping the \(hW\)-coordinate near a fixed point and varying the \(W\)-coordinate arbitrarily far contradicts \eqref{Eq:ABBoundedOrthogonalCoordinates}.

If \(W\pitchfork hW\), coarse surjectivity and unboundedness of \(\calC W\) allow us to choose \(x\) such that
\[\dist_{\calC W}\bigl(\pi_W(x),\rho_W^{hW}\bigr)>E\quad\text{and}\quad
\dist_{\calC(hW)}\bigl(h^\diamond(W)\pi_W(x),\rho_{hW}^W\bigr)>E+D_0.\]
Indeed, the two sets to avoid in \(\calC W\) are bounded. Transverse consistency gives
\[\dist_{\calC(hW)}\bigl(\pi_{hW}(x),\rho_{hW}^W\bigr)\le E.\]
This contradicts the second inequality and \eqref{Eq:ABBoundedOrthogonalCoordinates}. Thus \(hW=W\).

Now \eqref{Eq:ABBoundedOrthogonalCoordinates}, coarse surjectivity of \(\pi_W\), and the uniform bound on projection diameters give the required displacement bound for every \(p\in\calC W\). This bound applies to every power of \(h\), since those powers also belong to \(G_U^{\mathrm{AB}}\). If \(U\in\B(g)\), every other domain of \(\B(g)\) would be unbounded and orthogonal to \(U\), where all powers of \(g\) have uniformly bounded displacement. Hence \(\B(g)=\{U\}\), and \(g\) is fully supported on \(U\).
\end{proof}

The conclusion is bounded displacement, rather than exact triviality of the coordinate action. Thus geometric fiber preservation alone does not imply membership in the subgroup \(G_U\) of \Cref{Def:OurGUandFU}.

The \(\mathbf F_U\) stabilizers property of \cite[Definition~3.3]{AB23} is formulated using chosen geometric \(\mathbf F_U\)-fibers for a finite set of orbit representatives.
It implies that, for every \(U\), an orbit of \(G_U^{\mathrm{AB}}\) is coarsely equal to a chosen geometric \(\mathbf F_U\)-fiber.
Our property requires the natural induced quasi-action of \(G_U\) on \((\mathbf F_U,\frS_U)\) to yield an induced HHG structure on \(G_U\). The following lemma gives the comparison under these respective hypotheses.

\begin{lemma}[Finite-index comparison of stabilizers]\label{Lem:ABFiniteIndex}
For \(U\in\frS\), put \(J_U=G_U^{\mathrm{AB}}\cap G_U\).
\begin{enumerate}
\item If \((G,\frS)\) satisfies our \(\mathbf F_U\) stabilizers property, then \([G_U^{\mathrm{AB}}:J_U]<\infty\).
\item If \((G,\frS)\) satisfies the \(\mathbf F_U\) stabilizers property of \cite[Definition~3.3]{AB23}, then \([G_U:J_U]<\infty\).
\end{enumerate}
In particular, if both properties hold, then \(G_U^{\mathrm{AB}}\) and \(G_U\) are commensurable.
\end{lemma}

\begin{proof}
Fix a geometric \(\mathbf F_U\)-fiber \(F\subseteq\mathcal X\) and a point \(x\in F\), and write \(A=G_U^{\mathrm{AB}}\) and \(R=G_U\). Since \(A\) preserves \(F\), we have \(Ax\subseteq F\). 
By \Cref{Lem:OrthogonalStabilizerOrbit}, the orbit \(Rx\) lies in a bounded neighbourhood of \(F\). Under our stabilizers property, it is coarsely equal to \(F\).

Under our property, there is therefore a constant \(D\) such that, for every \(a\in A\), some \(r\in R\) satisfies \(\dist_{\mathcal X}(ax,rx)\le D\). Properness of the ambient action implies that \(r^{-1}a\) belongs to a fixed finite subset of \(G\). Thus \(A\) meets only finitely many right cosets of \(R\), and intersecting those cosets with \(A\) gives \([A:A\cap R]<\infty\).

Under the property of \cite{AB23}, the orbit \(Ax\) is coarsely equal to \(F\). Since \(Rx\) lies in a bounded neighbourhood of \(F\), the same argument with \(A\) and \(R\) interchanged gives \([R:A\cap R]<\infty\).
\end{proof}

\begin{corollary}\label{Cor:ABToRigidSupport}
Suppose that \((G,\frS)\) satisfies our \(\mathbf F_U\) stabilizers property. If \(g\in G_U^{\mathrm{AB}}\) and \(U\in\B(g)\), then some positive power of \(g\) is rigidly fully supported on \(U\).
\end{corollary}
\begin{proof}
By \Cref{Lem:ABFiniteIndex}, some positive power \(g^n\) belongs to \(G_U^{\mathrm{AB}}\cap G_U\). Since \(U\in\B(g^n)\), \Cref{Lem:RigidSupportCharacterization} applies.
\end{proof}

\begin{remark}[Passing to the subgroups \(G_U\)]\label{Rem:ComparisonWithABConti}
Suppose that an HHG structure in the class \(\Xi\) of \cite{AB23} also satisfies our \(\mathbf F_U\) stabilizers property. We show that this structure belongs to our class \(\Xi\), as asserted in \Cref{Rem:ComparisonWithAB}.

Let \(g\) have infinite order. After passing to a positive power, \cite[Definition~3.7]{AB23} gives
\[g^n=a_1\cdots a_k,\qquad a_i\in G_{U_i}^{\mathrm{AB}},\qquad \B(g)=\{U_1,\dots,U_k\},\]
where \(g^n\) fixes each \(U_i\) and acts loxodromically on \(\calC U_i\). The factors commute by \cite[Definition~3.8]{AB23}. By \Cref{Lem:ABFiniteIndex}, a common positive integer \(q\) satisfies \(a_i^q\in G_{U_i}^{\mathrm{AB}}\cap G_{U_i}\) for every \(i\). Consequently,
\[g^{nq}=a_1^q\cdots a_k^q.\]
For \(j\ne i\), the definition of \(G_{U_j}\) implies that \(a_j^q\) fixes \(U_i\) and acts trivially on \(\calC U_i\). Hence \(a_i^q\) acts there as \(g^{nq}\), and is loxodromic. By \Cref{Lem:RigidSupportCharacterization}, it is rigidly fully supported on \(U_i\). This proves our orthogonal decomposition property and hence membership in our class \(\Xi\).

This argument uses the exact commutation of the decomposition factors supplied by \cite{AB23}. It does not assert exact commutation of the whole subgroups \(G_U\) and \(G_V\) for \(U\bot V\).
\end{remark}

Our stabilizers property explicitly requires cofiniteness of \(G_U\curvearrowright\frS_U\). The geometric coarse-equality condition in \cite[Definition~3.3]{AB23} gives a proper and cobounded action of \(G_U^{\mathrm{AB}}\) on a chosen geometric \(\mathbf F_U\)-fiber, and hence an induced HHS structure on \(G_U^{\mathrm{AB}}\). In \cite[Lemma~3.4]{AB23}, this is asserted to define an HHG structure with index set \(\frS_U\). The additional point requiring justification is cofiniteness of \(G_U^{\mathrm{AB}}\curvearrowright\frS_U\): cofiniteness of \(G\curvearrowright\frS\) does not imply cofiniteness of this restricted action. The example below shows that the geometric stabilizers property does not supply this missing implication.

To distinguish this issue from the presence of inessential domains, we first recall the following reduction, proved in \cite[Lemma~4.2]{ABR25}; we use the HHG formulation recorded in \cite[Lemma~6.1]{OP25}.

\begin{lemma}[Passage to the essential domains, \cite{ABR25,OP25}]\label{Lem:EssentialReduction}
Let \((G,\frS)\) be an HHG, and let \(\frS_{\mathrm{ess}}\) be the collection of \(U\in\frS\) such that there exists an unbounded domain \(V\sqsubseteq U\). If \(\frS_{\mathrm{ess}}\neq\emptyset\), then \((G,\frS_{\mathrm{ess}})\) is an HHG whose hyperbolic spaces, projections, and hierarchical relations are inherited from \((G,\frS)\). Moreover, every \(\sqsubseteq\)-minimal domain in \(\frS_{\mathrm{ess}}\) is unbounded.
If \(\frS_{\mathrm{ess}}=\emptyset\), then \(G\) is finite.
\end{lemma}

\begin{remark}[Essential reduction and orthogonal stabilizers]\label{Rem:EssentialReductionAndStabilizers}
Suppose that \(\frS_{\mathrm{ess}}\neq\emptyset\). Essential reduction retains every unbounded domain, its coordinate space and action, and all hierarchical relations between retained domains. Consequently, for each \(U\in\frS_{\mathrm{ess}}\), the subgroup \(G_U\) in \Cref{Def:OurGUandFU} is unchanged. Bigsets, full support, rigid full support, and the expanded core graph are also unchanged. Thus our use of unbounded orthogonal coordinates is compatible with essential reduction.

This reduction removes only those domains having no unbounded domain nested into them. In particular, bounded domains containing unbounded subdomains are retained. The following example shows that cofiniteness of the induced stabilizer action can fail even when every domain is essential.
\end{remark}

\begin{example}\label{Ex:BoundedDomainPathology}
Let \(G=\mathbb Z^2=\langle a,b\rangle\), and use \(G\), equipped with the word metric associated to \(\{a,b\}\), as the ambient space.
Begin with the standard product structure, with maximal domain \(S\) and orthogonal domains \(U,V\sqsubsetneq S\), where
\[\calC S=\{*\},\qquad\calC U=\calC V=\mathbb R,\qquad
\pi_U(a^pb^q)=\{p\},\qquad\pi_V(a^pb^q)=\{q\}.\]
Adjoin domains \(T\) and \(W_n\), for \(n\in\mathbb Z\), with one-point coordinate spaces, and extend nesting by
\[V\sqsubsetneq W_n\sqsubsetneq T\sqsubsetneq S\qquad(n\in\mathbb Z).\]
Declare \(U\bot T\), and hence \(U\bot W_n\) and \(U\bot V\), and declare distinct \(W_n\) transverse. 
Projections to the new domains and bounded relative projection sets involving them have one-point targets. Take all downward relative projection maps from one-point spaces to \(\calC V=\mathbb R\) to have value \(\{0\}\).

These data define an HHS. Complexity is at most four, and the consistency and relative-projection consistency inequalities involving a new domain are immediate from the one-point coordinate spaces. The large links axiom for each \(W_n\) and for \(T\) is satisfied using \(V\); for \(S\), use \(U\) and \(T\). Bounded geodesic image is automatic whenever the upper coordinate space is a point. Within \(S\), the domain \(T\) is a container for the domains orthogonal to \(U\), and \(U\) is a container for the domains orthogonal to any of \(V,T,W_n\). There are no other container requirements, and these containers are clean. For partial realization, the only unbounded coordinates to prescribe are those of \(U\) and \(V\); all upper and transverse coordinates associated to a new domain are points. Uniqueness follows from the original product coordinates.

Let \(aW_n=W_{n+1}\), and let \(b\) fix every \(W_n\). Both generators fix \(S,U,V,T\). As in the original product structure, \(a\) translates \(\calC U\) by one and acts trivially on \(\calC V\), while \(b\) fixes \(\calC U\) pointwise and translates \(\calC V\) by one. These actions preserve the hierarchical relations and make the projections and bounded relative projection sets equivariant. There are five domain orbits, so this defines an HHG structure.

Choose the cosets of \(\langle a\rangle\) in \(G\) as geometric fibers for \(\mathbf F_U\), the cosets of \(\langle b\rangle\) for \(\mathbf F_V,\mathbf F_T,\mathbf F_{W_n}\), and \(G\) itself for \(\mathbf F_S\). The geometric stabilizers and the subgroups from \Cref{Def:OurGUandFU} agree:
\[G_S^{\mathrm{AB}}=G_S=G,\qquad G_U^{\mathrm{AB}}=G_U=\langle a\rangle,\]
\[G_D^{\mathrm{AB}}=G_D=\langle b\rangle\qquad\text{for }D\in\{V,T\}\cup\{W_n:n\in\mathbb Z\}.\]
For \(W_n\), its domain stabilizer is already \(\langle b\rangle\). For \(V\) and \(T\), trivial action on the orthogonal unbounded coordinate \(\calC U\) gives the same subgroup. Choosing fibers through the identity for orbit representatives shows that the \(\mathbf F_U\) stabilizers property of \cite{AB23} holds.

However,
\[\frS_T=\{T,V\}\cup\{W_n:n\in\mathbb Z\},\qquad G_T=\langle b\rangle,\]
and \(G_T\) fixes every \(W_n\). Thus the induced action on \(\frS_T\) has infinitely many orbits, so it does not define an HHG structure. This shows that the cofiniteness conclusion in \cite[Lemma~3.4]{AB23} does not follow from the geometric stabilizers property. Every domain in this example is essential, since each of \(W_n,T,S\) contains the unbounded domain \(V\). Hence essential reduction does not change the example.
\end{example}

\begin{remark}[The complementary factor in \cite{AB23}]\label{Rem:ABLemma34Repair}
The finite-order argument for the complementary factor in \cite[Lemma~3.9]{AB23} can be carried out in the ambient HHG, without the induced-HHG assertion of \cite[Lemma~3.4]{AB23}. Consider the case in which a complementary factor is present, and retain the notation of that argument:
\[\B(h)=\{H_1,\dots,H_k\},\qquad h=h_{H_1}\cdots h_{H_k}h_C,\]
where \(C\) is the clean container of the active domains, \(h_{H_i}\in G_{H_i}^{\mathrm{AB}}\) for every \(i\), and \(h_C\in G_C^{\mathrm{AB}}\). The factors commute by the commutative property of \cite{AB23}.

Suppose that \(h_C\) has infinite order, and choose \(W\in\B(h_C)\) using the axial--elliptic dichotomy in the ambient HHG \(G\). The \(\langle h_C\rangle\)-orbit of a point in a geometric \(\mathbf F_C\)-fiber remains in that fiber, whose projections outside \(\frS_C\) are bounded. Hence \(W\sqsubseteq C\), so \(W\bot H_i\) for every \(i\). By \Cref{Lem:ABBoundedOrthogonalAction}, every power of each \(h_{H_i}\) fixes \(W\) and moves every point of \(\calC W\) by at most the same constant \(D\).

Choose \(m>0\) such that \(h_C^m\) fixes \(W\) and acts loxodromically on \(\calC W\). Commutation implies that \(h^m\) fixes \(W\) and that, for every \(n\ge1\) and \(p\in\calC W\),
\[\dist_{\calC W}(p,h^{mn}p)\ge\dist_{\calC W}(p,h_C^{mn}p)-kD.\]
The right-hand side tends to infinity, so \(W\in\B(h)\), contradicting \(W\bot H_i\) for every \(i\). Thus \(h_C\) has finite order. By \cite[Theorem~G]{HHP23}, the ambient HHG has finitely many conjugacy classes of finite subgroups, so there is an integer \(K>0\), independent of \(h\), such that \(h_C^K=1\). Consequently,
\[h^K=h_{H_1}^K\cdots h_{H_k}^K.\]
This proves the finite-order conclusion and gives a power decomposition with the complementary factor removed. It does not require \((G_C^{\mathrm{AB}},\frS_C)\) to be an HHG and does not establish that induced-HHG assertion.
\end{remark}

In \Cref{Ex:BoundedDomainPathology}, removing \(T\) and all the \(W_n\) recovers the standard product structure without changing any unbounded coordinate data. This suggests allowing more general changes to the bounded domains.

\begin{question}\label{Qn:BoundedDomainModification}
Let \((G,\frS)\) satisfy the \(\mathbf F_U\) stabilizers property of \cite[Definition~3.3]{AB23}. Can one modify the bounded domains and their associated data, retaining the unbounded domains, their coordinate spaces, projections and actions, and their mutual hierarchical relations, to obtain an HHG structure \((G,\mathfrak T)\) which still satisfies that property and for which every geometric factor stabilizer \(G_U^{\mathrm{AB}}\) in the modified structure has finitely many orbits on \(\mathfrak T_U\)?
\end{question}

\printbibliography

\end{document}